\documentclass[a4paper]{amsart}

\usepackage[english]{babel}
\usepackage[utf8x]{inputenc}
\usepackage[T1]{fontenc}

\usepackage{amsthm,amssymb,amsmath,amscd}
\usepackage{graphicx}
\usepackage[colorinlistoftodos]{todonotes}
\usepackage[colorlinks=true, allcolors=blue]{hyperref}
\usepackage{mathtools}
\usepackage{algorithm}
\usepackage[noend]{algpseudocode}
\usepackage{enumerate}
\usepackage[version=3]{mhchem}

\newcommand{\R}{\mathbb{R}}
\newcommand{\Zint}{\mathbb{Z}}

\newcommand{\veq}{\mathrel{\reflectbox{\rotatebox[origin=c]{90}{$=$}}}}
\newcommand{\vsimeq}{\mathrel{\reflectbox{\rotatebox[origin=c]{90}{$\simeq$}}}}
\newcommand{\desc}{\textrm{desc}}
\newcommand{\odesc}{\overline{\textrm{desc}}}
\newcommand{\X}{\mathbb{X}}
\newcommand{\Y}{\mathbb{Y}}
\newcommand{\alp}{\textrm{Alp}}
\newcommand{\del}{\textrm{Del}}
\newcommand{\bsigma}{\bar{\sigma}}

\newtheorem{theorem}{Theorem}
\newtheorem{definition}[theorem]{Definition}
\newtheorem{cond}[theorem]{Condition}
\newtheorem{prop}[theorem]{Proposition}
\newtheorem{fact}[theorem]{Fact}
\newtheorem{lem}[theorem]{Lemma}
\newenvironment{claim}[1]{\par\noindent\underline{Claim:}\space#1}{}

\title[Volume Optimal Cycle]{Volume Optimal Cycle: Tightest representative cycle of 
a generator on persistent homology}
\author{Ippei Obayashi}

\begin{document}
\maketitle

\begin{abstract}
This paper shows a mathematical formalization, algorithms and computation software of volume optimal cycles, which are useful to understand geometric features
shown in a persistence diagram.
Volume optimal cycles give us concrete and optimal homologous structures,
such as rings or cavities, on a given data.
The key idea is the optimality on $(q + 1)$-chain complex for a $q$th homology generator. This optimality formalization is suitable for persistent homology. We can solve the optimization problem using linear programming.
For an alpha filtration on $\R^n$, volume optimal cycles on an $(n-1)$-th persistence diagram is more efficiently computable
using merge-tree algorithm.
The merge-tree algorithm also gives us a tree structure on the diagram and the structure has richer information. The key mathematical idea is Alexander duality.
\end{abstract}

\section{Introduction}

Topological Data Analysis (TDA)~\cite{carlsson,eh}, which clarifies the geometric features
of data from the viewpoint of topology, is developed rapidly in this century
both in theory and application. In TDA, persistent homology and its persistence
diagram (PD) \cite{elz,zc} are
important tool for TDA. Persistent homology enables us to capture
multiscale topological features effectively and quantitatively.
Fast computation softwares for persistent homology
are developed \cite{dipha,phat} and many applications are achieved such as
materials science \cite{Hiraoka28062016,granular,PhysRevE.95.012504},
sensor networks \cite{sensor}, evolutions of virus~\cite{virus}, and so on.
From the viewpoint of data analysis, a PD has some significant properties:
translation and rotation invariance, multiscalability and robustness to noise.
PDs are considered to be compact descriptors for complicated geometric data.

$q$th homology $H_q$ encodes $q$ dimensional geometric structures of data
such as connected components ($q=0$), rings ($q=1$), cavities ($q=2$), etc.
$q$th persistent homology encodes the information
about $q$ dimensional geometric structures with their scale.
A PD, a multiset\footnote{A multiset is a set with multiplicity on each point.}
in $\R\times(\R \cup \{\infty\})$, is used to summarize the information.
Each point in a PD is called a birth-death pair, which represents a homologous
structure in the data, and the scale is encoded on x-axis and y-axis.

\begin{figure}[hbtp]
  \centering
  \includegraphics[width=0.5\hsize]{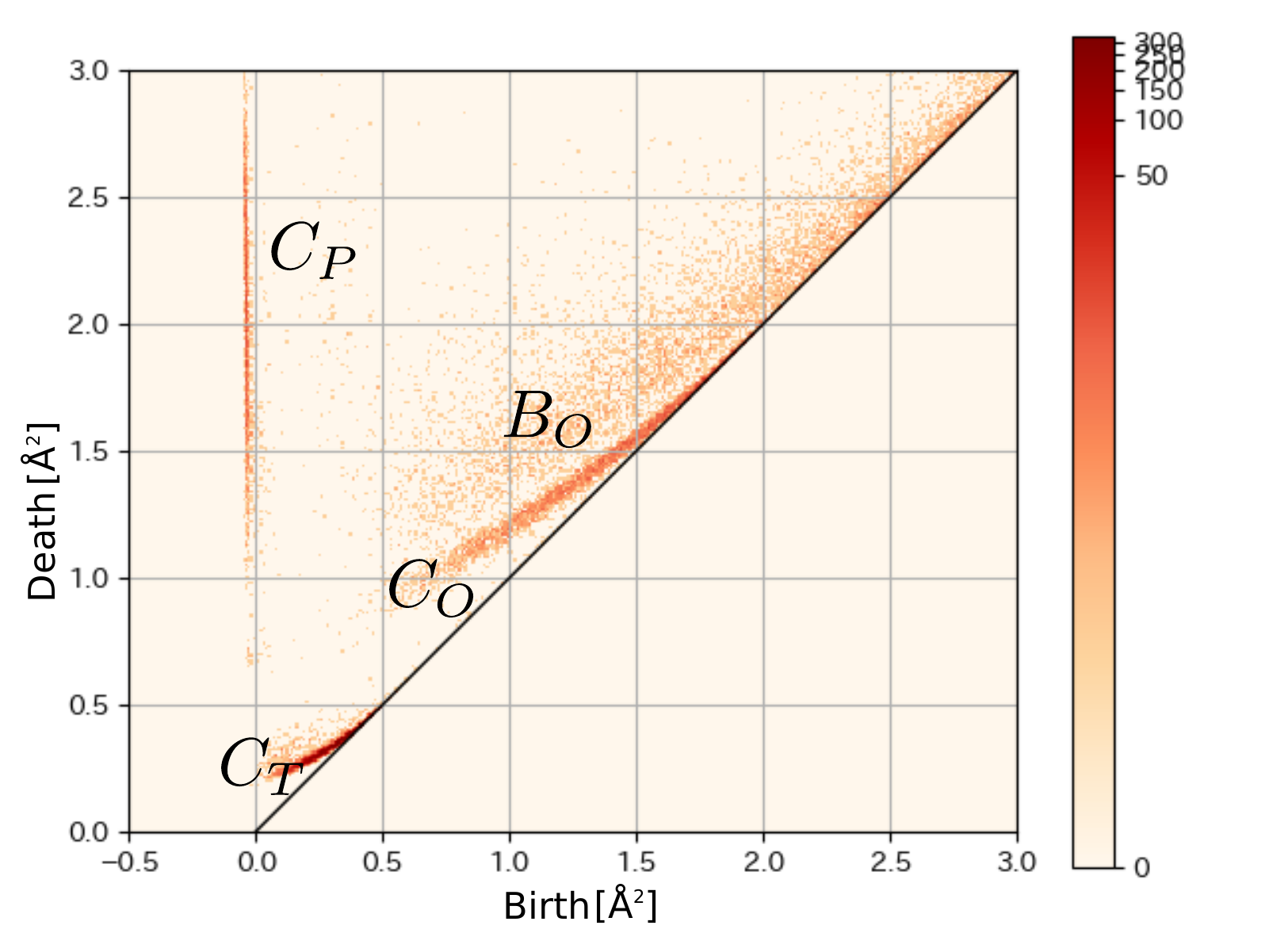}
  \caption{The 1st PD for the atomic configuration of amorphous silica
    in \cite{Hiraoka28062016},
    reproduced from the simulation data. The data is provided by Dr. Nakamura.
  }
  \label{fig:amorphous-silica}
\end{figure}

Typical workflow of the data analysis with persistent homology is as follows:
\begin{enumerate}
\item Construct a filtration from data
  \begin{itemize}
  \item Typical input data is a point cloud, a finite set of points in $\R^n$ and
    a typical filtration is an alpha filtration
  \end{itemize}
\item Compute the PD from the filtration
\item Analyze the PD to investigate the geometric features of the data
\end{enumerate}
In the last part of the above workflow, we often want to inversely reconstruct
a geometric structure corresponding each birth-death pair on the PD,
such as a ring or a cavity,
into the original input data.
Such inverse analysis is practically important for the
use of PDs. For example, we consider the 1st PD shown in Fig.~\ref{fig:amorphous-silica}
from the atomic configuration of amorphous silica
computed by molecular dynamical simulation
\cite{Hiraoka28062016}.
In this PD, there are some characteristic bands $C_P, C_T, C_O, B_O$, and these 
bands correspond to typical geometric structures in amorphous silica.
To analyze the PD more deeply, we want to reconstruct rings corresponding
such birth-death pairs in the original data. In the paper,
optimal cycles, one of such inverse analysis methods, are effectively used
to clarify such typical structures.

\begin{figure}[htbp]
  \centering
  \includegraphics[width=\hsize]{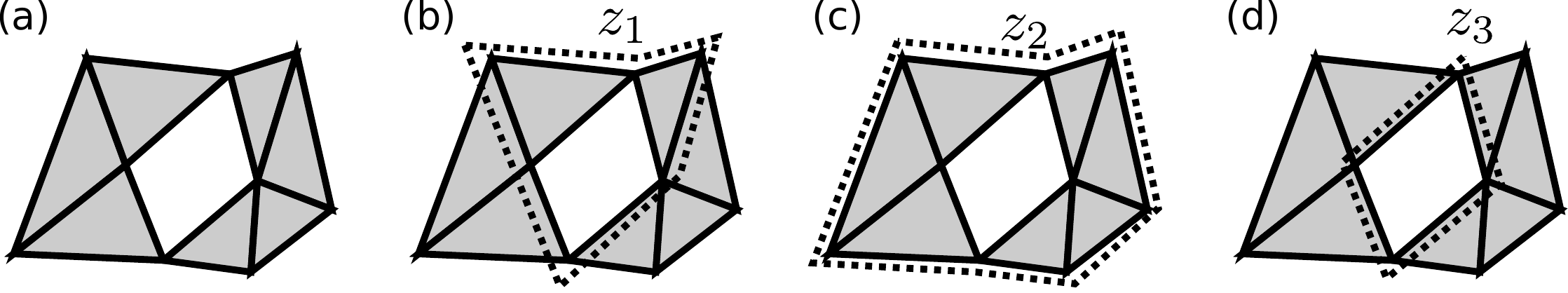}
  \caption{A simplicial complex with one hole.}
  \label{fig:optcyc_one_hole}
\end{figure}
A representative cycle of a generator of the homology vector space has such information,
but it is not unique and we want to find a better cycle
to understand the homology generator
for the analysis of a PD. 
For example, Fig.~\ref{fig:optcyc_one_hole}(a) has one homology generator on
$H_1$, and cycles $z_1$, $z_2$, and $z_3$ shown in Fig.~\ref{fig:optcyc_one_hole}
(b), (c), and (d)
are the same homologous information.
However, we consider that $z_3$ is best to understand the homology.
Optimization problems on homology are used to find such a representative cycle.
We can find the ``tightest'' representative cycle under a certain formalization.
Such optimization problems have been widely studied
under various settings\cite{optimal-Day,Erickson2005,Chen2011}, and
two concepts, optimal cycles\cite{Escolar2016} and
volume optimal cycles\cite{voc}, have been successfully
applied to persistent homology.
The optimal cycle minimizes the size of the cycle, while the volume optimal
cycle minimizes the internal volume of the cycle. Both of these two methods
give a tightest cycle in different senses.
The volume optimal cycles for persistent homology have been proposed in
\cite{voc} under the restriction of dimension. We can use them only for
$(n-1)$-th persistent homology embedded in $\R^n$, but
under the restriction,
there is an efficient computation algorithm using Alexander duality.

In this paper, we generalize the concept of volume optimal cycles on
any persistent homology and show the computation algorithm.
The idea in \cite{voc} is not applied to find a volume optimal ring
(a volume optimal cycle for $q=1$)
in a point cloud in $\R^3$ but our method is applicable to such a case.
In that case, optimal cycles are also applicable, but our new algorithm is
simpler, faster for large data, and gives us better information.

The contributions of this paper are as follows:
\begin{itemize}
\item The concept of volume optimal cycles is proposed to identify
  good representatives of generators in persistent homology.
  This is useful to understand a persistence diagram.
  \begin{itemize}
  \item The concept has been already proposed in \cite{voc} in a strongly limited sense
    about dimension and this paper
    generalize it.
  \item Optimal cycles are also usable for the same purpose, but the algorithm
    in this paper is easier to implement, faster, and gives
    better information.
    \begin{itemize}
    \item Especially, children birth-death pairs shown in Section \ref{sec:compare}
      are available only with volume optimal cycles.
    \end{itemize}
  \end{itemize}
\item Mathematical properties of volume optimal cycles are clarified.
\item Effective computation algorithms for volume optimal cycles are proposed.
\item The algorithm is implemented and some examples are computed by
  the program to show the usefulness of volume optimal cycles.
\end{itemize}

The rest of this paper is organized as follows. The fundamental ideas such as
persistent homology and simplicial complexes are introduced in Section~\ref{sec:ph}.
In Section~\ref{sec:oc} the idea of optimal cycles is reviewed.
Section~\ref{sec:voc} is the main part
of the paper. The idea of volume optimal cycles and the computation algorithm
in a general setting are presented in Section~\ref{sec:voc}.
Some mathematical properties of volume optimal cycles are also shown in this section.
In Section~\ref{sec:vochd} we show some special properties of 
$(n-1)$-th persistent homology in $\R^n$ and the faster algorithm.
We also explain tree structures in $(n-1)$-th persistent homology.
In Section~\ref{sec:compare}, we compare volume optimal cycles and optimal cycles.
In Section~\ref{sec:example} we show some computational examples by the proposed
algorithms. In Section~\ref{sec:conclusion}, we conclude the paper.

\section{Persistent homology}\label{sec:ph}

In this section, we explain some preliminaries about persistent homology
and geometric models. Persistent homology is available on
various general settings, but we mainly focus on the persistent homology
on a filtration of simplicial complexes, especially an alpha filtration
given by a point cloud.

\subsection{Persistent homology}

Let $\X = \{X_t \mid t \in T\}$ be a \textit{filtration} of topological spaces
where $T$ is a subset of $\Zint$ or $\R$.
That is, $X_t \subset X_{t'}$ holds for every $t \leq t'$.
Then we define $q$th homology
vector spaces $\{H_q(X_t)\}_{t \in T}$ whose coefficient is a field $\Bbbk$ and
homology maps $\varphi_s^t : H_q(X_s) \to H_q(X_t)$ for all $s \leq t$ induced by
inclusion maps $X_s \xhookrightarrow{} X_t$.
The family $H_q(\X) = (\{H_q(X_t)\}_t, \{ \varphi_s^t\}_{s \leq t})$
is called the $q$th \textit{persistent homology}.
The theory of persistent homology enables us to analyze the
structure of this family.

Under some assumptions, 
$H_q(\X)$ is uniquely decomposed as follows~\cite{elz,zc}:
\begin{align*}
  H_q(\X) = \bigoplus_{i=1}^p I(b_i, d_i),
\end{align*}
where $b_i \in T, d_i \in T \cup \{\infty\}$ with $b_i < d_i$.
Here, $I(b, d) = (U_t, f_s^t)$
consists of a family of vector spaces and linear maps:
\begin{align*}
  U_t&=\left\{
  \begin{array}{ll}
    \Bbbk, &\mbox{if } b \leq t < d, \\
    0, & \mbox{otherwise},
  \end{array}
         \right. \\
  f_s^t&:U_s \to U_t \\
  f_s^t&=\left\{
  \begin{array}{ll}
    \textrm{id}_\Bbbk, &\mbox{if } b \leq s \leq t < d, \\
    0, & \mbox{otherwise}.
  \end{array}
         \right. 
\end{align*}
This means that for each $I(b_i, d_i)$ there is 
a $q$ dimensional hole in $\X$ and it appears at $t = b_i$, persists up to $t < d_i$ and
disappears at $t = d_i$. In the case of $d_i = \infty$,
the $q$ dimensional hole never disappears on $\X$.
$b_i$ is called a \textit{birth time}, $d_i$ is called a
\textit{death time}, and the pair $(b_i, d_i)$ is called a \textit{birth-death pair}.
When $\X$ is a filtration of finite simplicial/cell/cubical complexes on $T$
with $\#T < \infty$ (we call $\X$ a \textit{finite filtration} under the condition),
such a unique decomposition exists.

When we have the unique decomposition,
the $q$th \textit{persistence diagram} of $\X$, $D_q(\X)$, is defined by a multiset
\begin{align*}
  D_q(\X) = \{(b_i, d_i) \mid i=1,\ldots, p\},
\end{align*}
and the 2D scatter plot or the 2D histogram of $D_q(\X)$ is often used to visualize
the diagram.

We investigate the detailed
algebraic structure of persistent homology for the preparation.
For simplicity, we assume the following condition on $\X$.
\begin{cond}\label{cond:ph}
  Let $X = \{\sigma_1, \ldots, \sigma_K\}$ be a finite simplicial complex.
  For any $1 \leq k \leq K$, $X_k = \{\sigma_1, \ldots, \sigma_k\}$ is
  a subcomplex of $X$.
\end{cond}
Under the condition,
\begin{align}
  \X: \emptyset = X_0 \subset X_1 \subset \cdots \subset X_K = X,\label{eq:ph}
\end{align}
is a filtration of complexes. For a general finite filtration, we can construct
a filtration satisfying Condition~\ref{cond:ph} by ordering all simplices properly.
Let $\partial_q: C_q(X) \to C_{q-1}(X)$ be the
boundary operator on $C_q(X)$ and 
$\partial_q^{(k)}: C_q(X_k) \to C_{q-1}(X_k)$ be a boundary operator of $C_q(X_k)$.
Cycles $Z_q(X_k)$ and boundaries $B_q(X_k)$ are defined by the
kernel of $\partial_q^{(k)}$ and the image of $\partial_{q+1}^{(k)}$, and $q$th homology
vector spaces are defined by $H_q(X_k) = Z_q(X_k)/B_q(X_k)$. Condition~\ref{cond:ph}
says that if $\sigma_k $ is a $q$-simplex,
\begin{equation}
  \label{eq:chain_plus1}
  \begin{aligned}
    C_q(X_{k}) & = C_q(X_{k-1})\oplus\left<\sigma_k\right>, \\
    C_{q'}(X_{k}) & = C_{q'}(X_{k-1}), \mbox{ for $q' \not = q$},
  \end{aligned}
\end{equation}
holds.
From the decomposition theorem and \eqref{eq:chain_plus1},
for each birth-death pair $(b_i, d_i)$,
we can find $z_i \in C_q(X)$ such that
\begin{align}
  &z_i \not \in Z_q(X_{b_i-1}), \label{eq:birth_pre}\\
  &z_i  \in Z_q(X_{b_i}) = Z_q(X_{b_i-1}) \oplus \left<\sigma_{b_i}\right>, \label{eq:birth_post}\\
  &z_i  \not \in B_q(X_{k}) \mbox{ for $k < d_i$}, \label{eq:death_pre} \\
  &z_i  \in B_q(X_{d_i}) = B_q(X_{d_i-1}) \oplus \left<\partial \sigma_{d_i}\right>, \label{eq:death_post} \\
  &\{[z_i]_k  \mid b_i \leq k < d_i\} \text{ is a basis of } H_q(X_k), \label{eq:ph-basis}
\end{align}
where $[z]_k = [z]_{B_q(X_k)} \in H_q(X_k)$.
\eqref{eq:death_post} holds only if $d_i \not = \infty$. This $[z_i]_k$ is a
homology generator that persists from $k={b_i}$ to $k = {d_i-1}$.
$\{z_i\}_{i=1}^p$ is called the \textit{persistence cycles} for
$D_p(\X) = \{(b_i, d_i)\}_{i=1}^p$.
An algorithm of computing a PD actually finds persistence cycles from a given
filtration.
The persistence cycle of $(b_i, d_i)$ is not unique, therefore
we want to find a ``good'' persistence cycle 
to find out the geometric structure corresponding to each birth-death pair.
That is the purpose of 
the volume optimal cycle, which is the main topic of this paper.
We remark that the condition \eqref{eq:ph-basis} can be easily proved from
(\ref{eq:birth_pre}-\ref{eq:death_post}) and the decomposition theorem,
and
hence we only need to show (\ref{eq:birth_pre}-\ref{eq:death_post}) to prove
that given $\{z_i\}_{i=1}^p$ are persistence cycles.

\subsection{Alpha filtration}

One of the most used filtrations for data analysis using persistent homology
is an alpha filtration~\cite{eh, em}. An alpha filtration is defined from a point cloud,
a set of finite points $P = \{x_i \in \R^n\}$.
The alpha filtration is defined as a filtration of alpha complexes and they are
defined by a Voronoi diagram and a Delaunnay triangulation. 

The \textit{Voronoi diagram} for a point cloud $P$, which is a decomposition of $\R^n$ into
\textit{Voronoi cells} $\{V(x_i) \mid x_i \in P\}$, is defined by
\begin{align*}
  V(x_i) = \{x \in \R^n \mid \|x - x_i\|_2 \leq \|x - x_j\|_2 \text{ for any } j\not = i\}.
\end{align*}
The \textit{Delaunnay triangulation} of $P$, $\del(P)$, which is a simplicial complex
whose vertices are points in $P$, is defined by
\begin{align*}
  \del(P) = \{[x_{i_1} \cdots x_{i_q}] \mid
  V(x_{i_1}) \cap \cdots \cap V(x_{i_q}) \not = \emptyset\},
\end{align*}
where $[x_{i_0} \cdots x_{i_q}]$ is the $q$-simplex whose vertices are
$x_{i_0}, \ldots, x_{i_q} \in P$.
Under the assumption of general position in the sense of \cite{em},
the Delaunnay triangulation is a simplicial decomposition of
the convex hull of $P$ and it has good geometric properties.
The \textit{alpha complex} $\alp(P, r)$ with radius parameter $r \geq 0$,
which is a subcomplex of $\del(P)$, is defined as follows:
\begin{align*}
  \alp(P, r) = \{[x_{i_0} \cdots x_{i_q}] \in \del(P) \mid
  B_r(x_{i_0}) \cap \cdots \cap B_r(x_{i_q}) \not = \emptyset \},
\end{align*}
where $B_r(x)$ is the closed ball whose center is $x$ and whose radius is $r$.
A significant property of the alpha complex is the following homotopy equivalence
to the $r$-ball model.
\begin{align*}
  \bigcup_{x_i \in P} B_r(x_i) \simeq |\alp(P, r)|,
\end{align*}
where $|\alp(P,r)|$ is the geometric realization of $\alp(P, r)$.
The \emph{alpha filtration} for $P$ is defined by $\{\alp(P,r)\}_{r\geq 0}$.
Figure~\ref{fig:alpha} illustrates an example of a filtration by $r$-ball model
and the corresponding alpha filtration. The 1st PD of this filtration is
$\{(r_2, r_5), (r_3, r_4)\}$.
Since there are $r_1 < \cdots < r_K$ such
that $\alp(P, s) = \alp(P, t)$ for any $r_i \leq s < t < r_{i+1}$, we can
treat the alpha filtration as a finite filtration
$\alp(P, r_1) \subset \cdots \subset \alp(P, r_K)$.

\begin{figure}[htbp]
  \centering
  \includegraphics[width=0.8\hsize]{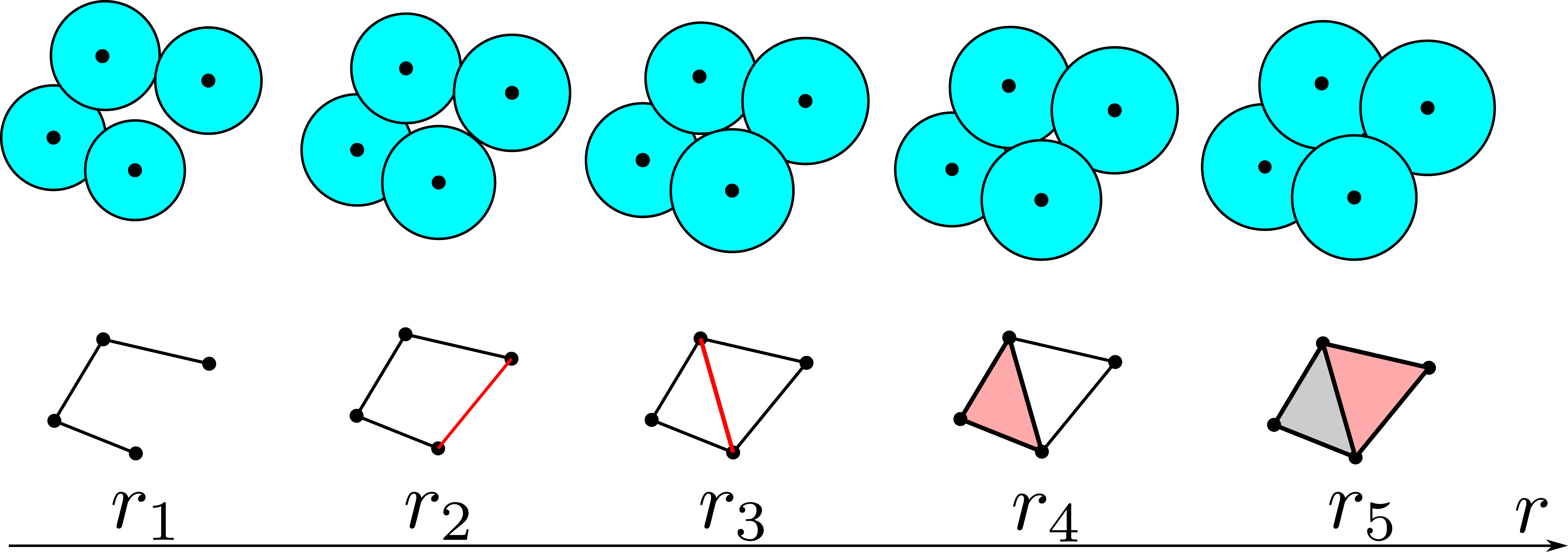}
  \caption{An $r$-ball model and the corresponding alpha filtration. Each red simplex
    in this figure appears at the radius parameter $r_i$.
  }
  \label{fig:alpha}
\end{figure}

We mention an weighted alpha complex and its filtration~\cite{weightedalpha}.
An alpha complex is topologically equivalent to the union of $r$-balls,
while an weighted alpha complex is topologically equivalent to the union
of $\sqrt{r^2+\alpha_i}$-balls, where $\alpha_i$ depends on each point.
The weighted alpha filtration is useful to study the geometric structure
of a point cloud whose points have their own radii. For example,
for the analysis of atomic configuration,
the square of ionic radii or Van der Waals radii are used as $\alpha_i$. 

\section{Optimal cycle}\label{sec:oc}

First, we discuss an optimal cycle on normal homology whose
coefficient is $\Bbbk = \Zint_2$.
Figure~\ref{fig:optcyc_one_hole}(a) shows a simplicial complex whose
1st homology vector space $H_1$ is isomorphic to $\Zint_2$.
In Fig.~\ref{fig:optcyc_one_hole}(b), (c), and (d),
$z_1$, $z_2$, and $z_3$ have same information about $H_1$. That is,
$H_1 = \left<[z_1]\right> = \left<[z_2]\right> = \left<[z_3]\right>$. However,
we intuitively consider that $z_3$ is the best to represent the hole in
Fig.~\ref{fig:optcyc_one_hole} since $z_3$ is the shortest loop in these
loops. 
Since the size of a loop $z = \sum_{\sigma:1-\text{simplex}} \alpha_\sigma\sigma \in Z_1(X)$
is equal to
\begin{align*}
  \#\{\sigma : 1\text{-simplex} \mid \alpha_\sigma \not = 0 \},
\end{align*}
and this is $\ell^0$ ``norm''\footnote{
  For a finite dimensional $\R$- or $\mathbb{C}$- vector space
  whose basis is $\{g_i\}_i$,
  the $\ell^0$ norm $\|\cdot\|_0$ is defined by
  $\|\sum_i \alpha_i g_i \|_0 = \# \{i \mid \alpha_i \not = 0 \}$.
  Mathematically this is not a norm since it is not homogeneous, but
  in information science and statistics, it is called $\ell^0$ norm.
}\footnote{
  On a $\Zint_2$-vector space, any norm is not defined mathematically, but
  it is natural that we call this $\ell^0$ norm.
},
we write it $\|z\|_0$.
Here, $z_3$ is the solution of the following problem:
\begin{align*}
  \mbox{minimize } \|z\|_0 ,\mbox{ subject to } z\sim z_1.
\end{align*}
The minimizing $z$ is called the \textit{optimal cycle} for $z_1$.
From the definition of homology, we can rewrite the problem as follows:
\begin{equation}
  \label{eq:optcyc_one_hole}
  \begin{aligned}
    \mbox{minimize } &\|z\|_0, \mbox{ subject to:} \\
    z &= z_1 + \partial w, \\
    w &\in C_2(X).
  \end{aligned}
\end{equation}
Now we complete the formalization of the optimal cycle
on a simplicial complex with one hole.

\begin{figure}[tbp]
  \centering
  \includegraphics[width=0.4\hsize]{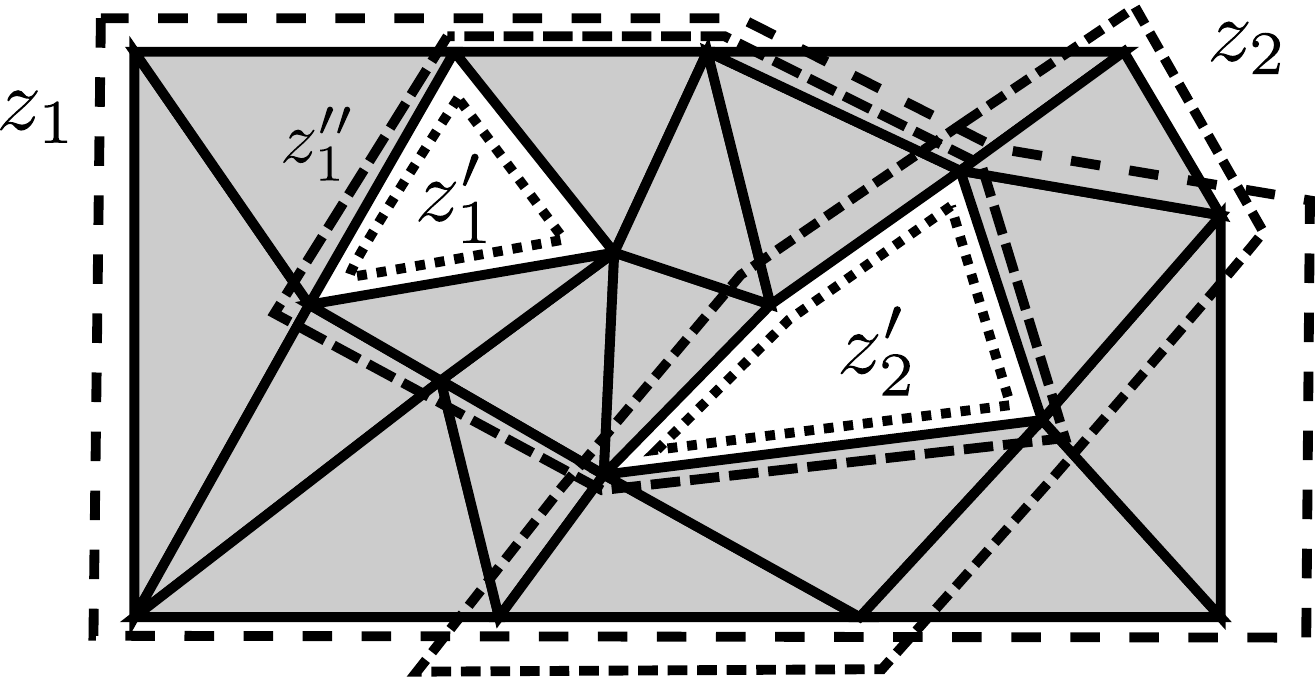}
  \caption{A simplicial complex with two holes.}
  \label{fig:optcyc_two_hole}
\end{figure}

How about the case if a complex has two or more holes?
We consider the example in Fig.~\ref{fig:optcyc_two_hole}.
From $z_1$ and $z_2$, we try to find $z_1'$ and $z_2'$ using a similar
formalization. If we apply the optimization \eqref{eq:optcyc_one_hole}
to each $z_1$ and $z_2$, $z_1''$ and $z_2'$ are found. How can we
find $z_1'$ from $z_1$ and $z_2$?
The problem is a hole represented by $z_2'$, therefore we ``fill''
that hole and solve the minimization problem. Mathematically,
filling a hole corresponds to considering $Z_1(X)/(B_1(X) \oplus \left<z_2'\right>)$ instead
of $Z_1(X)/B_1(X)$ and 
the following optimization problem gives us the required loop $z_1'$.
\begin{align*}
  \mbox{minimize } &\|z\|_0, \mbox{ subject to:} \\
  z & = z_1 + \partial w + k z_2, \\
  w &\in C_2(X), \\
  k & \in \Zint_2.
\end{align*}

When you have a complex that has many holes, you can apply the idea
repeatedly to find all optimal cycles. The idea of optimal cycles
obviously applied $q$th homology for any $q$.

\subsection{How to compute an optimal cycle}\label{subsec:fast-computation}

Finding a basis of a homology vector space is not a difficult problem for a computer.
We prepare a matrix representation of the boundary operator and
apply matrix reduction algorithm. Please read \cite{comphom} for the detailed algorithm.
Therefore the problem is how to solve the above minimizing problem.

In general, solving a optimization problem on a $\Zint_2$ linear space is a
difficult problem. The problem is a kind of
combinatorial optimization problems. They are well studied but it is
well known that such a problem is sometimes hard to solve on a computer.

One approach is using linear programming, used in \cite{sensor-l0-l1}.
Since optimization problem on $\Zint_2$ is
hard, we use $\R$ as a coefficient. For $\R$ coefficient, $\ell^0$ norm also 
means the size of loop and $\ell^0$ optimization is natural for our purpose.
However, $\ell^0$ optimization is also a difficult problem. Therefore we replace
$\ell^0$ norm to $\ell^1$ norm. It is well known in the fields of sparse sensing and
machine learning that
$\ell^1$ optimization gives a good approximation of $\ell^0$
optimization.
That is, we solve the following optimization problem
instead of \eqref{eq:optcyc_one_hole}.
\begin{equation}
  \label{eq:optcyc_one_hole-l1}
  \begin{aligned}
    \mbox{minimize } &\|z\|_1,  \mbox{ subject to:} \\
    z &= z_1 + \partial w, \\
    w &\in C_2(X; \R).
  \end{aligned}
\end{equation}
This is called a linear programming and we can solve the problem very efficiently
by good solvers such as cplex\footnote{\url{https://www-01.ibm.com/software/commerce/optimization/cplex-optimizer/}} and Clp\footnote{\url{https://projects.coin-or.org/Clp}}.

Another approach is using integer programming, used in \cite{optimal-Day,Escolar2016}.
$\ell^1$ norm optimization gives a good approximation, but maybe the solution is not
exact. However, if all coefficients are restricted into $0$ or $\pm 1$
in the optimization problem \eqref{eq:optcyc_one_hole-l1},
the $\ell^0$ norm and $\ell^1$ norm is identical, and it gives a better solution.
This restriction on the coefficients has another advantage that
we can understand the optimal solution in more intuitive way.
Such an optimization problem is
called integer programming. Integer programming is much slower than linear programming,
but
some good solvers such as cplex and Clp are available for integer programming.


\subsection{Optimal cycle for a filtration}

Now, we explain optimal cycles on a filtration to
analyze persistent homology shown in \cite{Escolar2016}.
We start from the example Fig.~\ref{fig:optcyc_filtration}.

\begin{figure}[htbp]
  \centering
  \includegraphics[width=\hsize]{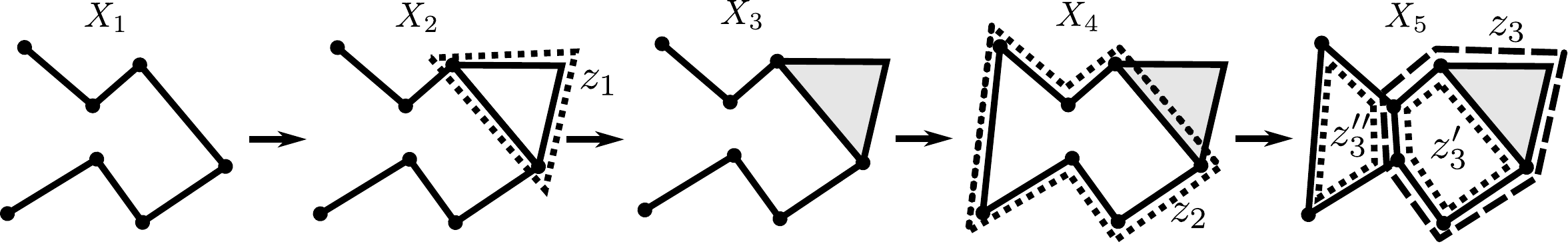}
  \caption{A filtration example for optimal cycles.}
  \label{fig:optcyc_filtration}
\end{figure}

In the filtration, a hole $[z_1]$ appears at $X_2$ and disappear at $X_3$,
another hole $[z_2]$ appears at $X_4$ and $[z_3]$ appears at $X_5$.
The 1st PD of the filtration is $\{(2,3), (4,\infty), (5, \infty)\}$.
The persistence cycles
$z_1, z_2, z_3$, are computable by the algorithm of persistent homology and
we want to find $z_3'$ or $z_3''$ to analyze the hole corresponding to
the birth-death pair $(5, \infty)$.
The hole $[z_1]$ has been already dead at $X_5$ and $[z_2]$ remains alive at $X_5$,
so we can find $z_3'$ or $z_3''$ to solve the following optimization problem:
\begin{align*}
  \mbox{minimize } & \|z\|_0 \mbox{ subject to: } \\
  z &= z_3 + \partial w + k z_2, \\
  w & \in C_1(X_5), \\
  k & \in \Bbbk.
\end{align*}
In this case, $z_3''$ is chosen because $\|z_3'\|_0 > \|z_3''\|_0$.
By generalizing the idea, we show 
Algorithm~\ref{alg:optcyc} to find optimal cycles for a filtration $\X$\footnote{
  In fact, in \cite{Escolar2016}, two slightly different algorithms are shown,
  and this algorithm is one of them.
}. Of course, to solve the optimization problem in Algorithm~\ref{alg:optcyc},
we can use the computation techniques shown in Section~\ref{subsec:fast-computation}.

\begin{algorithm}[ht]
  \caption{Computation of optimal cycles on a filtration}\label{alg:optcyc}
  \begin{algorithmic}
    \State Compute $D_q(\X)$ and
    persistence cycles $z_1, \ldots, z_n$
    \State Choose $(b_i, d_i) \in D_q(\X)$ by a user
    \State Solve the following optimization problem
    \begin{align*}
      \mbox{minimize } &\|z\|_1, \mbox{ subject to:} \\
      z &= z_i + \partial w + \sum_{j \in T_i} \alpha_j z_j, \\
      w & \in C_q(X_{b_i}), \\
      \alpha_j & \in \Bbbk, \\
      \text{where } T_i& = \{j \mid b_j < b_i < d_j\}.
    \end{align*}
  \end{algorithmic}
\end{algorithm}

\section{Volume optimal cycle}\label{sec:voc}

In this section, we propose volume optimal cycles, a new tool to
characterize generators appearing in persistent homology. 
In this section, we will show the generalized version of volume optimal cycles and
the computation algorithm.
The limited version of volume optimal cycles shown in \cite{voc} will be explained
in the next section.

We assume Condition~\ref{cond:ph} and consider the filtration
$\X: \emptyset = X_0 \subset \cdots \subset X_K = X$.
A \textit{persistent volume} for $(b_i, d_i) \in D_q(\X)$
is defined as follows.
\begin{definition}
  $z \in C_{q+1}(X)$ is a persistent volume for $(b_i, d_i) \in D_q(\X)$
  if $z$ satisfies the following conditions:
  \begin{align}
    z &= \sigma_{d_i} + \sum_{\sigma_k \in \mathcal{F}_{q+1}} \alpha_k \sigma_k,
        \label{eq:vc-1}\\
    \tau^*(\partial z) &= 0 \mbox{ for all } \tau \in \mathcal{F}_{q}, \label{eq:vc-2}\\
    \sigma_{b_i}^*(\partial z) &\not =  0, \label{eq:vc-3}
  \end{align}
  where $\mathcal{F}_{q} = \{ \sigma_k : q\textup{-simplex} \mid b_i < k < d_i \}$,
  $\{ \alpha_k \in \Bbbk \}_{\sigma_k \in \mathcal{F}_{q+1}}$, and $\sigma_k^*$ is the
  dual basis of cochain $C^q(X)$, i.e. $\sigma_k^*$ is the linear map on $C_q(X)$
  satisfying $\sigma_k^*(\sigma_j) = \delta_{kj}$ for any $\sigma_k, \sigma_j$: $q$-simplex.
\end{definition}
Note that the persistent volume is defined only if the death time is finite.

The \textit{volume optimal cycle} for $(b_i, d_i)$
and the \textit{optimal volume} for the pair are defined as follows.
\begin{definition}\label{defn:voc}
  $\partial \hat{z}$ is the volume optimal cycle and
  $\hat{z}$ is the optimal volume for $(b_i, d_i) \in D_q(\X)$
  if $\hat{z}$ is the solution
  of the following optimization problem.
  \begin{center}
    minimize $\|z\|_0$, subject to \eqref{eq:vc-1}, \eqref{eq:vc-2}, and \eqref{eq:vc-3}.
  \end{center}
\end{definition}

The following theorem ensures that the optimization problem
of the volume optimal cycle always has a solution.

\begin{theorem}\label{thm:existence_voc}
  There is always a persistent volume of any $(b_i, d_i) \in D_q(\X)$.
\end{theorem}

The following theorem ensures that the volume optimal cycle
is good to represent the homology generator corresponding to $(b_i, d_i)$.

\begin{theorem}\label{thm:good_voc}
  Let $\{x_j \mid j=1, \ldots, p\}$ be all persistence cycles for $D_q(\X)$.
  If $z_i$ is a persistent volume of $(b_i, d_i) \in D_q(\X)$,
  $\{x_j \mid j\not = i\} \cup \{\partial z_i\}$ are also
  persistence cycles for $D_q(\X)$.
\end{theorem}

Intuitively say, a homology generator is dead by filling the internal volume of
a ring, a cavity, etc., and a persistent volume is such an internal volume.
The volume optimal cycle
minimize the internal volume instead of the size of the cycle.

\begin{proof}[Proof of Theorem \ref{thm:existence_voc}]
  Let $z_i$ be a persistence cycle satisfying (\ref{eq:birth_pre}-\ref{eq:death_post}).
  Since
  \begin{align*}
    z_i \in B_q(X_{d_i}) \backslash B_q(X_{d_i-1}),
  \end{align*}
  we can write $z_i$ as follows.
  \begin{equation}
    \begin{aligned}
      z_i &= \partial (w_0 + w_1), \\
      w_0 &= \sigma_{d_i} + \sum_{\sigma_k \in \mathcal{F}_{q+1}} \alpha_k \sigma_k,\\
      w_1 &= \sum_{\sigma_k \in \mathcal{G}_{q+1}} \alpha_k \sigma_k,
    \end{aligned}\label{eq:phbase_decomp}
  \end{equation}
  where $\mathcal{G}_{q+1} = \{\sigma_k: (q+1)\textrm{-simplex} \mid k < b_i\}$.
  Note that the coefficient of $\sigma_{d_i}$ in $w_0$ can be normalized as in
  \eqref{eq:phbase_decomp}.
  Now we prove that $w_0$ is a persistent volume.
  From $z_i \in Z_q(X_{b_i})$ and $\partial w_1 \in C_q(X_{b_i-1})$, we have
  $\partial w_0 = z_i - \partial w_1 \in C_q(X_{b_i})$ and this means that
  $\tau^*(\partial w_0) = 0$ for all $\tau \in \mathcal{F}_q$.
  From $\partial w_1 \in C_q(X_{b_i-1})$, we have $\sigma_{b_i}^*(\partial w_1) = 0$
  and therefore $\sigma_{b_i}^*(\partial w_0) = \sigma_{b_i}^*(z_i)$, and the right hand side
  is not zero since
  $z_i \in Z_q(X_{b_i}) \backslash Z_q(X_{b_i-1}) \subset C_q(X_{b_i})\backslash C_q(X_{b_i-1})$.
  Therefore
  $w_0$ satisfies all conditions (\ref{eq:vc-1}-\ref{eq:vc-3}).
\end{proof}

\begin{proof}[Proof of Theorem \ref{thm:good_voc}]
  We prove the following arguments.
  The theorem follows from these arguments.
  \begin{align*}
    \partial z_i &\in Z_q(X_{b_i}) \backslash Z_q(X_{b_i-1}), \\
    \partial z_i &\in B_q(X_{d_i}) \backslash B_q(X_{d_i-1}).
  \end{align*}
  The condition \eqref{eq:vc-2},
  $\tau^*(\partial z_i) = 0 \mbox{ for all } \tau \in \mathcal{F}_{q} $,
  means $\partial z_i \in Z_q(X_{b_i})$.
  The condition \eqref{eq:vc-3},
  $\sigma_{b_i}^*(\partial z_i) \not = 0$, means
  $\partial z_i \not \in Z_q(X_{b_i-1})$.
  Since $\partial z_i = \partial \sigma_{d_i} + \sum_{\sigma_k \in \mathcal{F}_{q+1}} \alpha_k \partial \sigma_k,$
  and $B_q(X_{d_i}) = B_q(X_{d_i - 1}) \oplus \left< \partial \sigma_{d_i} \right>$,
  we have $\partial z_i \in B_q(X_{d_i}) \backslash B_q(X_{d_i-1})$ and this finishes the
  proof.
\end{proof}


\subsection{Algorithm for volume optimal cycles}
To compute the volume optimal cycles, we can apply the same strategies as
optimal cycles. Using linear programming with $\R$ coefficient and
$\ell^1$ norm is efficient and
gives sufficiently good results. Using integer programming is slower, but it gives
better results.

Now we remark the condition \eqref{eq:vc-3}. In fact it is impossible
to handle this condition by linear/integer programming directly.
We need to replace this condition
to $|\sigma_{b_i}^*(\partial z)| \geq \epsilon$ for sufficiently small $\epsilon > 0$
and we need to solve the optimization problem twice
for $\sigma_{b_i}^*(\partial z) \geq \epsilon$ and
$\sigma_{b_i}^*(\partial z) \leq -\epsilon$. However, as mentioned later,
we can often remove the constraint \eqref{eq:vc-3} to solve the problem and
this fact is useful for faster computation.

We can also apply the following heuristic performance improvement technique
to the algorithm for an alpha filtration by using the locality of
an optimal volume.
The simplices which contained in the optimal volume for $(b_i, d_i)$,
are contained 
in a neighborhood of $\sigma_{d_i}$. Therefore we take a parameter $r > 0$, and
we use
$\mathcal{F}_q^{(r)} = \{\sigma \in \mathcal{F}_q \mid \sigma \subset B_r(\sigma_{d_i}) \}$
instead of $\mathcal{F}_q$ to reduce the size of
the optimization problem,
where $B_r(\sigma_{d_i})$ is the ball of radius $r$ whose center is the
centroid of $\sigma_{d_i}$.
Obviously, we cannot find a solution with a too small $r$.
In Algorithm~\ref{alg:volopt}, $r$ is properly chosen by a user but
the computation software 
can automatically increase $r$ when the optimization problem cannot find
a solution.

We also use another heuristic for faster computation.
To treat the constraint \eqref{eq:vc-3}, we need to apply linear programming twice
for positive case and negative case.
In many examples, the optimized solution automatically satisfies \eqref{eq:vc-3}
even if we remove the constrain.
There is an example in which the corner-cutting does not work (shown in
\ref{subsec:properties-voc}), but it works well in many cases.
One way is that we try to solve the linear programming without \eqref{eq:vc-3} and
check the \eqref{eq:vc-3}, and if \eqref{eq:vc-3} is satisfied, output the solution.
Otherwise, we solve the linear programming twice with \eqref{eq:vc-3}.

The algorithm to compute a volume optimal cycle for an alpha filtration is
Algorithm~\ref{alg:volopt}.

\begin{algorithm}[h!]
  \caption{Algorithm for a volume optimal cycle}\label{alg:volopt}
  \begin{algorithmic}
    \Procedure{Volume-Optimal-Cycle}{$\X, r$}
    \State Compute the persistence diagram $D_q(\X)$
    \State Choose a birth-death pair $(b_i, d_i) \in D_q(\X)$ by a user
    \State Solve the following optimization problem:
    \begin{align*}
      \mbox{minimize } &\|z\|_1, \mbox{ subject to:}\\
      z &= \sigma_{d_i} + \sum_{\sigma_k \in \mathcal{F}_{q+1}^{(r)}} \alpha_k \sigma_k, \\
      \tau^*(\partial z) &= 0 \mbox{ for all } \tau \in \mathcal{F}_{q}^{(r)}. \\      
    \end{align*}
    \If{we find the optimal solution $\hat{z}$}
      \If{$\sigma_{b_i}^*(\partial \hat{z}) \not = 0$}
        \State \Return $\hat{z}$ and $\partial \hat{z}$
        \Else
        \State Retry optimization twice with the additional constrain:
        \begin{align*}
          \sigma_{b_i}^*(\partial z) \geq \epsilon \text{ or }
          \sigma_{b_i}^*(\partial z) \leq -\epsilon
        \end{align*}
      \EndIf  
    \Else
    \State \Return the error message to the user to choose larger $r$.
    \EndIf
    \EndProcedure
  \end{algorithmic}
\end{algorithm}

If your filtration is not an alpha filtration, possibly
you cannot use the locality technique. However, in that case,
the core part of the algorithm works fine and you can use the algorithm.

\subsection{Some properties about volume optimal cycles}
\label{subsec:properties-voc}
In this subsection, we remark some properties about volume optimal cycles.

First, the volume optimal cycle for a birth-death pair
is not unique. Figure~\ref{fig:multiple-voc} shows such an example.
In this example, $D_1 = \{(1, 5), (3, 4), (2, 6)\}$ and
both (b) and (c) is the optimal volumes of the birth-death pair $(2, 6)$.
In this filtration, any weighted sum of (b) and (c) with weight $\lambda$
and $1-\lambda$ ($0 \leq \lambda \leq 1$)
in the sense of chain complex is the volume optimal cycle of $(2, 6)$
if we use $\R$ as a coefficient and $\ell^1$ norm.
However, standard linear programing algorithms
choose an extremal point solution, hence the algorithms choose either $\lambda=0$ or
$\lambda=1$ and our algorithm outputs either (b) or (c).

\begin{figure}[htbp]
  \centering
  \includegraphics[width=\hsize]{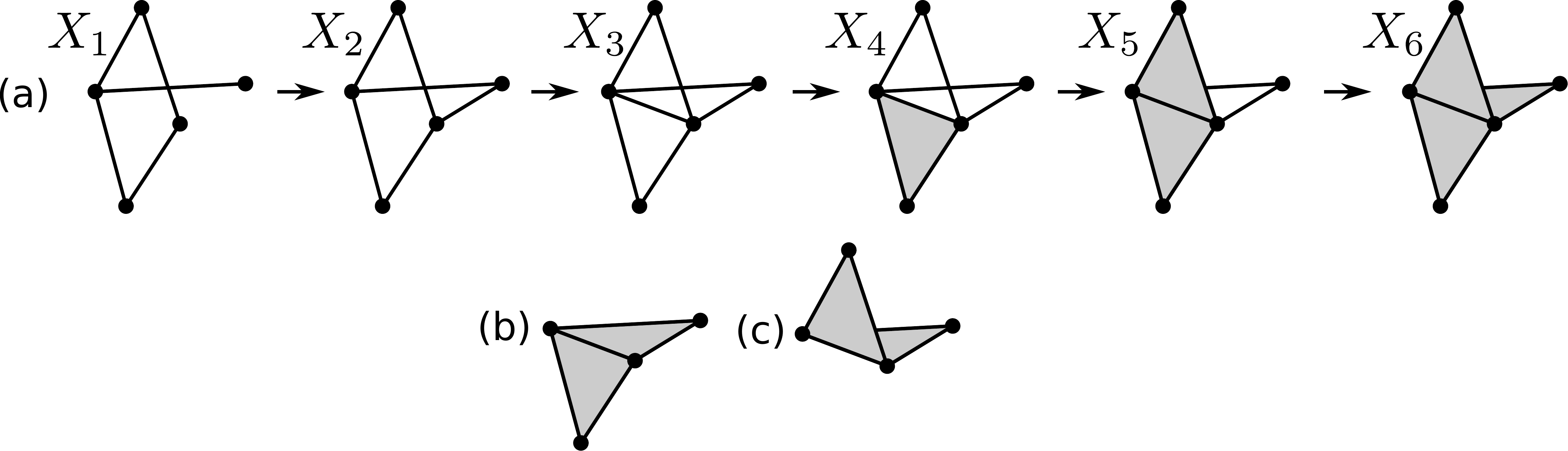}
  \caption{An example of non-unique volume optimal cycles.}
  \label{fig:multiple-voc}
\end{figure}

\begin{figure}[htbp]
  \centering
  \includegraphics[width=0.9\hsize]{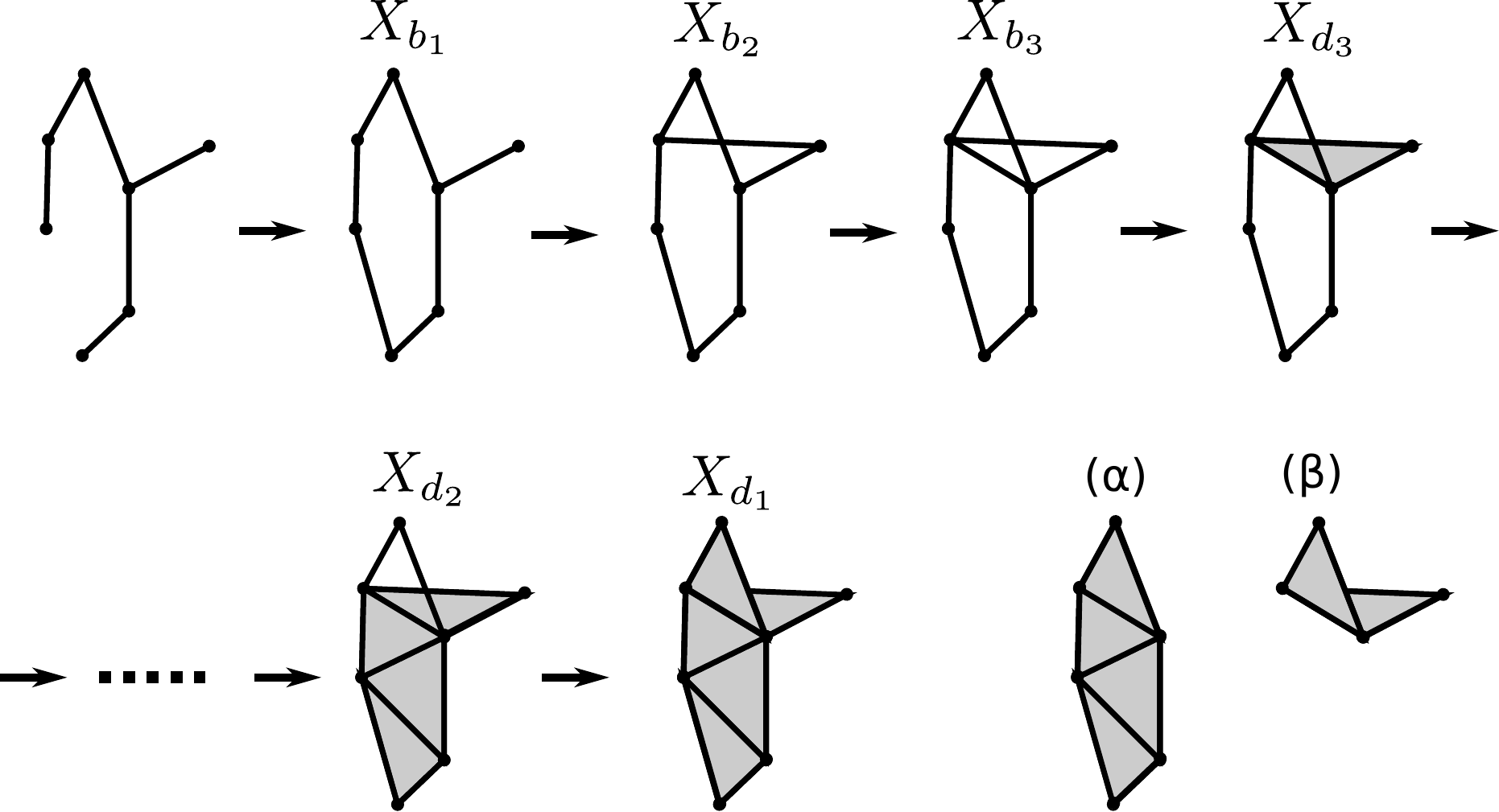}
  \caption{An example of the failure of the computation of the volume optimal cycle
  if the constrain \eqref{eq:vc-3} is removed.}
  \label{fig:voc-failure}
\end{figure}

Second, by the example in Fig~\ref{fig:voc-failure}, we show  
that the optimization problem for the volume optimal cycle may
give a wrong solution if the constrain \eqref{eq:vc-3} is removed.
In this example, $(b_1, d_1), (b_2, d_2), (b_3, d_3)$ are birth-death pairs
in the 1st PD, and the volume optimal cycle for $(b_1, d_1)$ is ($\alpha$) in
Fig.~\ref{fig:voc-failure}, but the algorithm gives ($\beta$) if
the constrain \eqref{eq:vc-3} is removed.



\section{Volume optimal cycle on $(n-1)$-th persistent homology}\label{sec:vochd}
 
In this section, we consider a triangulation of a convex set in $\R^n$ and
its $(n-1)$-th persistent homology. More precisely, we assume the following
conditions.
\begin{cond}\label{cond:rn}
  A simplicial complex $X$ in $\R^n$ 
  satisfies the following conditions.
  \begin{itemize}
  \item Any $k$-simplex $(k<n)$ in $X$ is a face of an $n$-simplex
  \item $|X|$ is convex
  \end{itemize}
\end{cond}
For example, an alpha filtration satisfies the above conditions
if the point cloud has more than $n$ points and
satisfies the general position condition. In addition, we assume
Condition~\ref{cond:ph} to simplify the statements of results and algorithms.

The thesis \cite{voc} pointed out that 
$(n-1)$-th persistent homology is
isomorphic to 0th persistent cohomology of the dual filtration by the Alexander duality
under the assumption.
By using this fact, the thesis defined volume optimal cycles under different formalization
from ours. The thesis defined a volume optimal cycle as an output of
Algorithm~\ref{alg:volopt-hd-compute}.
In fact, the two definitions of volume optimal cycles are equivalent
on $(n-1)$-th persistent homology.
0th persistent cohomology is deeply related to the connected components,
and we can compute the volume optimal cycle by linear computation cost.
The thesis also pointed out that $(n-1)$-th persistent homology has a tree structure called
persistence trees (or PH trees).

In this section, we always use $\Zint_2$ as a coefficient of homology
since using $\Zint_2$ makes the problem easier.

The following theorems hold.
\begin{theorem}\label{thm:vochd-unique}
  The optimal volume for $(b_i, d_i) \in D_{n-1}(\X)$ is uniquely determined.
\end{theorem}

\begin{theorem}\label{thm:vochd-tree}
  If $z_i$ and $z_j$ are the optimal volumes for two different birth-death
  pairs
  $(b_i, d_i)$ and
  $(b_j, d_j)$ in $D_{n-1}(\X)$, one of the followings holds:
  \begin{itemize}
  \item $z_i \cap z_j = \emptyset$,
  \item $z_i \subset z_j$,
  \item $z_i \supset z_j$.
  \end{itemize}
  Note that we can naturally regard any
  $z = \sum_{\sigma: n\text{-simplex}} k_{\sigma} \sigma \in C_{n}(X)$ as a subset of $n$-simplices of $X$,
  $\{\sigma : n\text{-simplex} \mid k_{\sigma} \not = 0\}$,
  since we use $\Zint_2$ as a homology coefficient.
\end{theorem}

From Theorem~\ref{thm:vochd-tree}, we know that
$D_{n-1}(\X)$ can be regarded as a forest (i.e. a set of trees)
by the inclusion relation. The trees are called \textit{persistence trees}.

We can compute all optimal volumes and persistence trees on $D_{n-1}(\X)$
by the merge tree algorithm (Algorithm~\ref{alg:volopt-hd-compute}).
This algorithm is a modified version of the algorithm in \cite{voc}.
To describe the algorithm,
we prepare a directed graph $(V, E)$ where $V$ is a set of nodes and
$E$ is a set of edges. In the algorithm, an element of $V$ is
a $n$-cell in $X \cup \{\sigma_{\infty}\}$ and an element of $E$ is
a directed edge between two $n$-cells, where $\sigma_\infty = \R^n \backslash X$
is the $n$-cell in the one point compactification space $\R^n \cup \{\infty\} \simeq S^n$. 
An edge has extra data
in $\Zint$ and we write the edge from $\sigma$ to $\tau$ with
extra data $k$ as $(\sigma \xrightarrow{k} \tau)$.
Since the graph is always a forest through the whole algorithm,
we always find a root
of a tree which contains a $n$-cell $\sigma$ in the graph $(V, E)$
by recursively following edges from $\sigma$.
We call this procedure \textproc{Root}($\sigma, V, E$).

\begin{algorithm}
  \caption{Computing persistence trees by merge-tree algorithm}\label{alg:volopt-hd-compute}
  \begin{algorithmic}
    \Procedure{Compute-Tree}{$\X$}
    \State initialize $V = \{\sigma_\infty\}$ and $E = \emptyset$
    \For{$k=K,\ldots,1$}
      \If{$\sigma_k$ is a $n$-simplex}
        \State add $\sigma_k$ to $V$
      \ElsIf{$\sigma_k$ is a $(n-1)$-simplex}
        \State let $\sigma_s$ and $\sigma_t$ are two $n$-cells
          whose common face is $\sigma_k$
        \State $\sigma_{s'} \gets \textproc{Root}(\sigma_s, V, E)$
        \State $\sigma_{t'} \gets \textproc{Root}(\sigma_t, V, E)$
        \If{$s'=t'$}
          \State \textbf{continue}
        \ElsIf{$s'> t'$}
          \State Add $(\sigma_{t'} \xrightarrow{k} \sigma_{s'})$ to $E$
        \Else
          \State Add $(\sigma_{s'} \xrightarrow{k} \sigma_{t'})$ to $E$
        \EndIf
      \EndIf
    \EndFor
    \Return $(V, E)$
    \EndProcedure
  \end{algorithmic}
\end{algorithm}

The following theorem gives us the interpretation of
the result of the algorithm to the persistence information.
\begin{theorem}\label{thm:vochd-alg}
  Let $(V, E)$ be a result of Algorithm~\ref{alg:volopt-hd-compute}. Then
  the followings hold.
  \begin{enumerate}[(i)]
  \item $D_{n-1}(\X) = \{(b, d) \mid (\sigma_d \xrightarrow{b} \sigma_s) \in E\}$
  \item The optimal volume for $(b, d)$ is
    all descendant nodes of $\sigma_d$ in $(E, V)$
  \item The persistence trees is computable from $(E, V)$. That is,
    $(b_i, d_i)$ is a child of $(b_j, d_j)$ if and only if there are edges
    $\sigma_{d_i} \xrightarrow{b_i} \sigma_{d_j} \xrightarrow{b_j} \sigma_{s}$.
  \end{enumerate}
\end{theorem}

The  theorems in this section can be proven from the following facts:
\begin{itemize}
\item From Alexander duality, for a simplicial complex $X$ in $\R^n$,
  \begin{align*}
    H_q(X) \simeq H^{n-q-1}((\R^n\backslash X)\cup\{\infty\}),
  \end{align*}
  holds.
  \begin{itemize}
  \item $\infty$ is required for one point compactification of $\R^n$.
  \item More precisely, we use the dual decomposition of $X$.
  \end{itemize}
\item By applying above Alexander duality to a filtration,
  $(n-1)$-th persistent homology is isomorphic to $0$-th persistent cohomology
  of the dual filtration.
\item On a cell complex $\bar{X}$, a basis of $0$-th cohomological vector space is
  given by
  \begin{align*}
    \{ \sum_{\sigma \in C} \sigma^* &\mid C \in \textrm{cc}(\bar{X})\},
  \end{align*}
  where $\textrm{cc}(\bar{X})$
  is the connected component decomposition of 0-cells in $\bar{X}$.
\item Merge-tree algorithm traces the change of connectivity in the filtration, and
  it gives the structure of 0-the persistent cohomology.
\end{itemize}

We prove the theorems in Appendix~\ref{sec:pfvochd}.

\subsection{Computation cost for merge-tree algorithm}
\label{sec:faster}

In the algorithm, we need to find the root from its descendant node.
The naive way to find the root is following the graph step by step
to the ancestors. In the worst case, the time complexity
of the naive way
is $O(N)$ where $N$ is the number of 
of $n$-simplices, and total time complexity of the algorithm becomes $O(N^2)$.
The union-find algorithm~\cite{unionfind} is used
for a similar data structure, and we can apply the idea of union-find algorithm.
By adding a shortcut path to the root in a similar way as the union-find algorithm,
the amortized time complexity is improved to almost constant time\footnote{
  More precisely, the amortized time complexity is bounded by the inverse of
  Ackermann function and it is less than 5 if
  the data size is less than $2^{2^{2^{2^{16}}}}$. Therefore we can regard the time
  complexity as constant.
}.
Using the technique, the total time complexity of the Algorithm~\ref{alg:volopt-hd-compute}
is $O(N)$.

\section{Comparison between volume optimal cycles
  and optimal cycles}\label{sec:compare}

In this section, we compare volume optimal cycles and optimal cycles.
In fact, optimal cycles and volume optimal cycles are identical in many cases.
However, since we can use optimal volumes in addition to volume optimal cycles,
we have more information than optimal cycles.
One of the most prominent advantage of volume optimal cycles is children birth-death pairs,
explained below.

\subsection{Children birth-death pairs}

In the above section, we show that there is a tree structure
on an $(n-1)$-th persistence diagram computed from
a triangulation of a convex set in $\R^n$. Unfortunately,
such a tree structure does not exist in a general case.
However, in the research of amorphous solids by persistent homology\cite{Hiraoka28062016},
a hierarchical structure of rings in $\R^3$ is effectively used, and
it will be helpful if we can find such a structure on a computer.
In \cite{Hiraoka28062016}, the hierarchical structure
was found by computing all optimal cycles and
searching multiple optimal cycles which have common vertices.
However, computing all optimal cycles or all volume optimal cycles
is often expensive as shown in Section \ref{subsec:performance} and
we require a cheaper method. The optimal volume is available for that purpose.
When the optimal volume for a birth-death pair $(b_i, d_i)$ is
$\hat{z} = \sigma_{d_i} + \sum_{\sigma_k \in \mathcal{F}_{q+1}} \hat{\alpha}_k \sigma_k$,
the \textit{children birth-death pairs} of $(b_i, d_i)$ is defined as follows:
\begin{align*}
  \{(b_j, d_j) \in D_q(\X) \mid \sigma_{d_j} \in \mathcal{F}_{q+1},
  \hat{\alpha}_{d_j} \not = 0 \}.
\end{align*}
This is easily computable from a optimal volume with low computation cost.

Now we remark that if we consider $(n-1)$-th persistent homology in $\R^n$,
the children birth-death pairs of $(b_i, d_i) \in D_{n-1}(\X)$ is identical to
all descendants of $(b_i, d_i)$ in the tree structure. This fact is known from
Theorem~\ref{thm:vochd-tree}. This fact suggests that we can use
children birth-death pairs as a good substitute for the tree structure appearing
on $D_{n-1}(\X)$ in $\R^n$. The ability of children birth-death pairs is shown in
Section \ref{sec:example-silica}, the example of amorphous silica.

\subsection{Some examples in which volume optimal cycles and
  optimal cycles are different}

We show some differences between optimal cycles and volume optimal cycles
on a filtration.
In Fig~\ref{fig:oc-voc-diff-1},
the 1st PD of this filtration is $\{(2, 5), (3, 4)\}$.
The optimal cycle of $(3, 4)$ is $z_1$ since 
$\|z_1\|_1 < \|z_2\|_1$ but the volume optimal cycle is $z_2$.
In this example, $z_2$ is better than $z_1$ to represent the
birth-death pair $(3, 4)$.
The example is deeply related to Theorem~\ref{thm:good_voc}.
Such a  theorem does not hold for optimal cycles and it means that an optimal cycle may
give misleading information about a birth-death pair.
This is one advantage of volume optimal cycles compared to optimal cycles.

\begin{figure}[htbp]
  \centering
  \includegraphics[width=0.9\hsize]{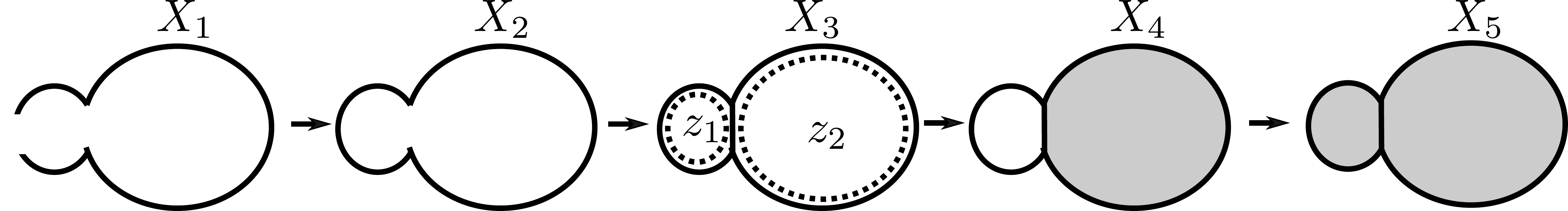}
  \caption{A filtration whose optimal cycle and volume optimal cycle are different.}
  \label{fig:oc-voc-diff-1}
\end{figure}

In Fig.~\ref{fig:oc-voc-diff-2} and Fig.~\ref{fig:oc-voc-diff-3},
optimal cycles and volume optimal cycles are also different.
In Fig.~\ref{fig:oc-voc-diff-2},
the optimal cycle is $z_1$ but the volume optimal cycle $z_2$.
In Fig.~\ref{fig:oc-voc-diff-3},
the optimal cycle for $(3, 4)$ is $z_1$ 
but the volume optimal cycle is $z_1 + z_2$.

\begin{figure}[htbp]
  \centering
  \includegraphics[width=0.3\hsize]{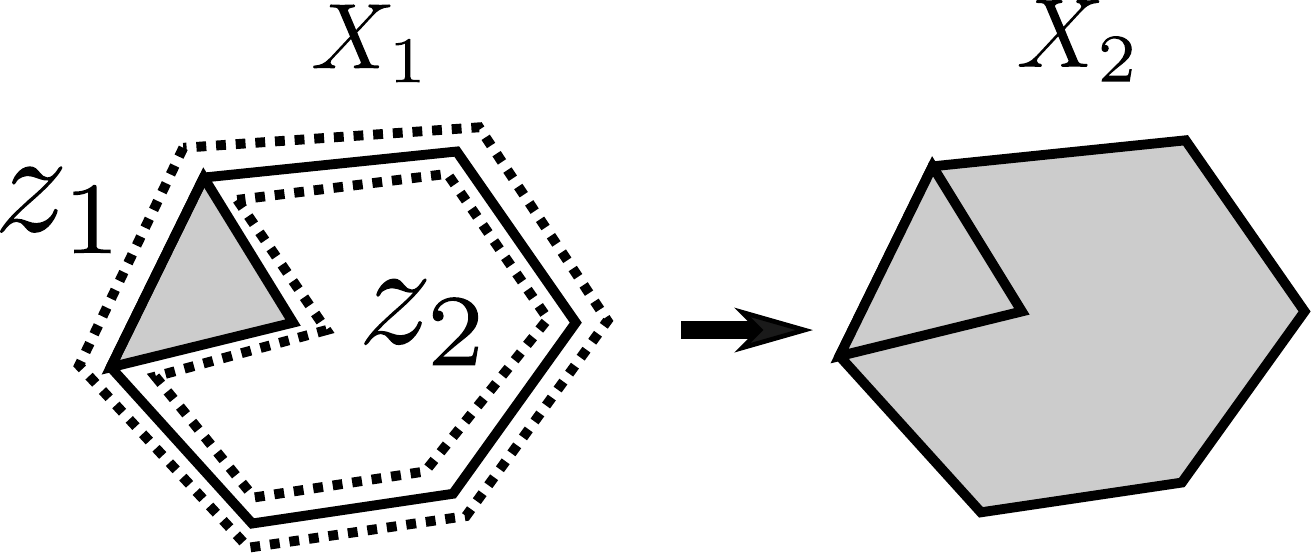}
  \caption{Another filtration whose optimal cycle and volume optimal cycle are different.}
  \label{fig:oc-voc-diff-2}
\end{figure}

\begin{figure}[htbp]
  \centering
  \includegraphics[width=0.7\hsize]{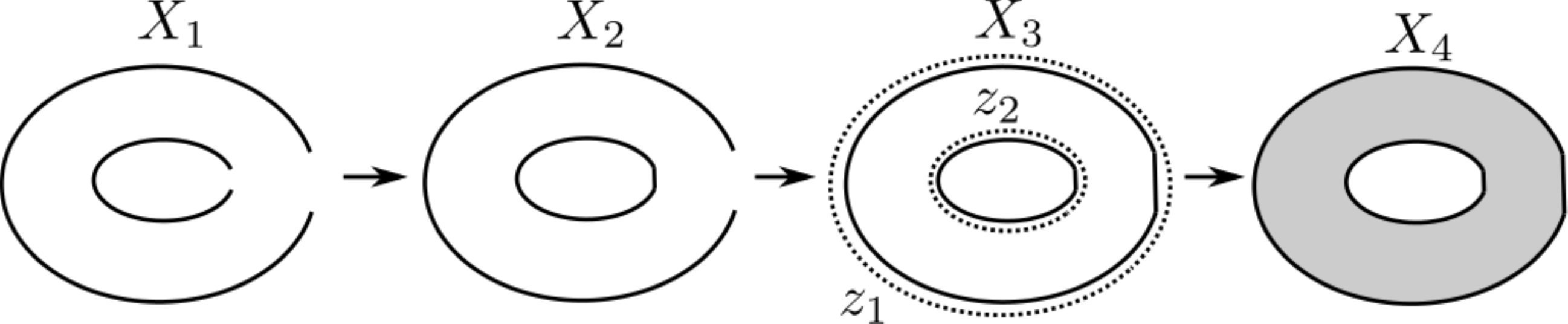}
  \caption{Another filtration whose optimal cycle and volume optimal cycle are different.}
  \label{fig:oc-voc-diff-3}
\end{figure}

In Fig~\ref{fig:no-voc}, the 1st PD is $(2, \infty)$ and we cannot define the volume optimal
cycle but can define the optimal cycle. In general, we cannot define
the volume optimal cycle for a birth-death pair with infinite death time.
If we use an alpha filtration in $\R^n$, such a problem doest not occur because
a Delaunnay triangulation is always acyclic. But if we use another type of a filtration,
we possibly cannot use volume optimal cycles.
That may be a disadvantage of volume optimal cycles if we use a filtration other than
an alpha filtration, such as a Vietoris-Rips filtration.
\begin{figure}[htbp]
  \centering
  \includegraphics[width=0.3\hsize]{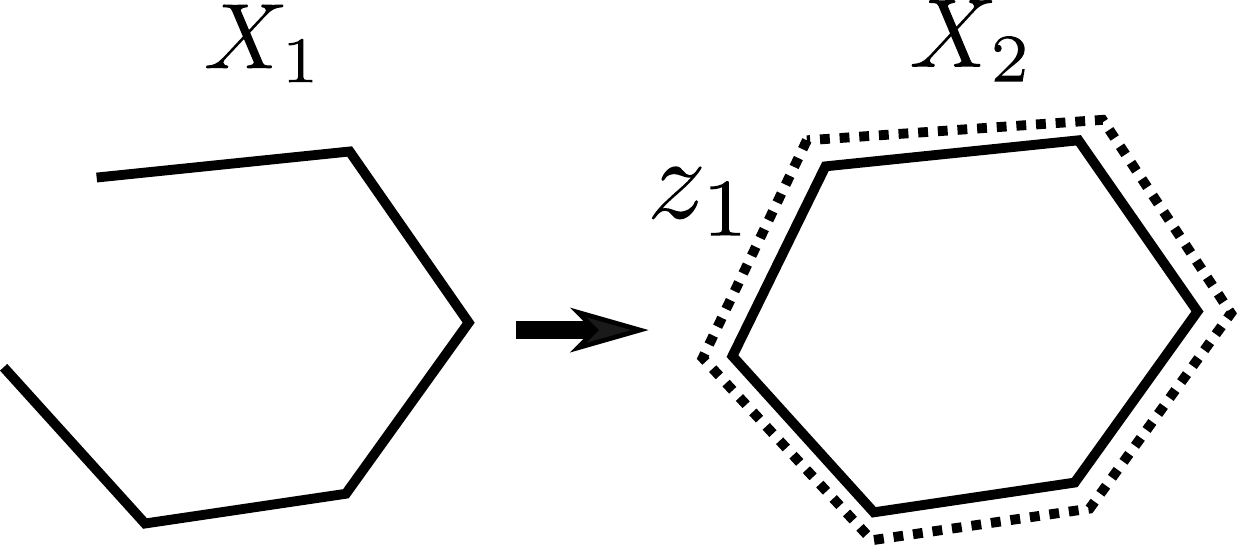}
  \caption{A filtration without a volume optimal cycle.}
  \label{fig:no-voc}
\end{figure}


One more advantage of the volume optimal cycles is
the simplicity of the computation algorithm. For the computation of
the optimal cycles we need to keep track of all persistence cycles but
for the volume optimal cycles we need only birth-death pairs.
Some efficient algorithms implemented in phat and dipha do not keep track of
such data, hence we cannot use such softwares to compute the optimal cycles
without modification.
By contrast we can use such softwares for the computation of
the volume optimal cycles.

\section{Example}\label{sec:example}

In this section, we will show the example results of our algorithm.
In all of these examples, we use alpha or weighted alpha filtrations.

For all of these examples, optimal volumes and volume optimal cycles are computed
on a laptop PC with 1.2 GHz Intel(R) Core(TM) M-5Y71 CPU and 8GB memory on Debian 9.1.
Dipha~\cite{dipha} is used to compute PDs,
CGAL\footnote{\url{http://www.cgal.org/}} is used to compute (weighted) alpha filtrations,
and Clp~\cite{coin} is used to solve the linear programming.
Python is used to write the program and pulp\footnote{\url{https://github.com/coin-or/pulp}} is used for
the interface to Clp from python.
Paraview\footnote{\url{https://www.paraview.org/}} is used to visualize volume optimal cycles.

If you want to use the software, please contact with
us. Homcloud\footnote{\url{http://www.wpi-aimr.tohoku.ac.jp/hiraoka_labo/research-english.html}},
a data analysis software with persistent homology developed
by our laboratory, provides the algorithms shown in this paper.
Homcloud provides the easy access to the volume optimal cycles. We can visualize
the volume optimal cycle of a birth-death pair only by clicking the pair in a PD on
Homcloud's GUI.

\subsection{2-dimensional Torus}

The first example is a 2-dimensional torus in $\R^3$. 2400 points are randomly scattered on
the torus and PDs are computed. Figure~\ref{fig:pd-torus} shows the
1st and 2nd PDs. The 1st PD has two birth-death pairs
$(0.001, 0.072)$ and $(0.001, 0.453)$ and the 2nd PD has
one birth-death pair $(0.008, 0.081)$ far from the diagonal. These birth-death pairs
correspond to generators of $H_1(\mathbb{T}^2) \simeq \Bbbk^2$ and
$H_2(\mathbb{T}^2) \simeq \Bbbk$.

\begin{figure}[tbp]
  \centering
  \includegraphics[width=0.4\hsize]{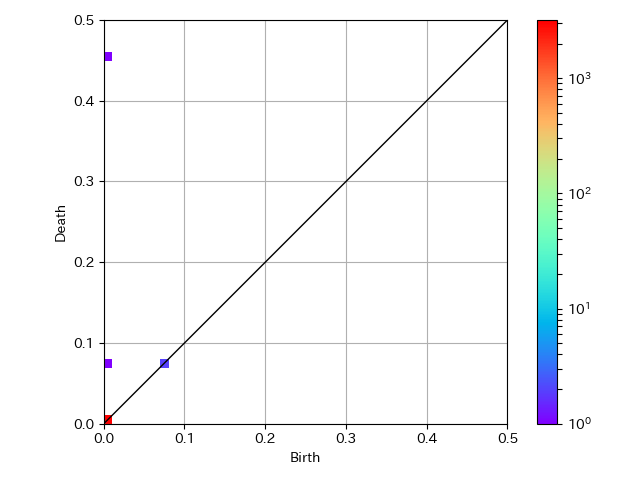}
  \includegraphics[width=0.4\hsize]{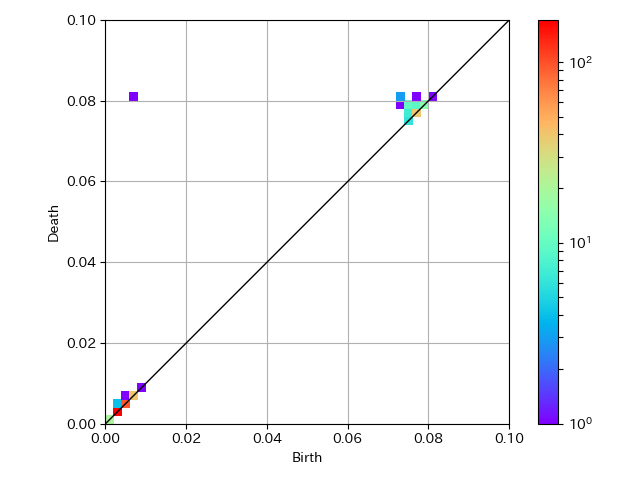}
  \caption{The 1st and 2nd PDs of the point cloud on a torus.}
  \label{fig:pd-torus}
\end{figure}

\begin{figure}[tbp]
  \centering
  \includegraphics[width=0.3\hsize]{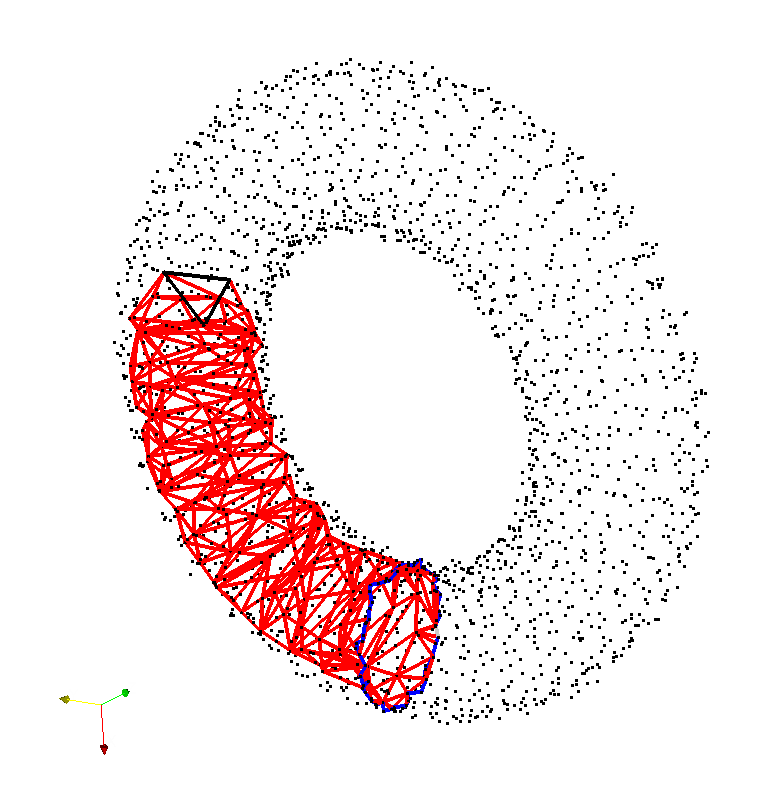}
  \includegraphics[width=0.3\hsize]{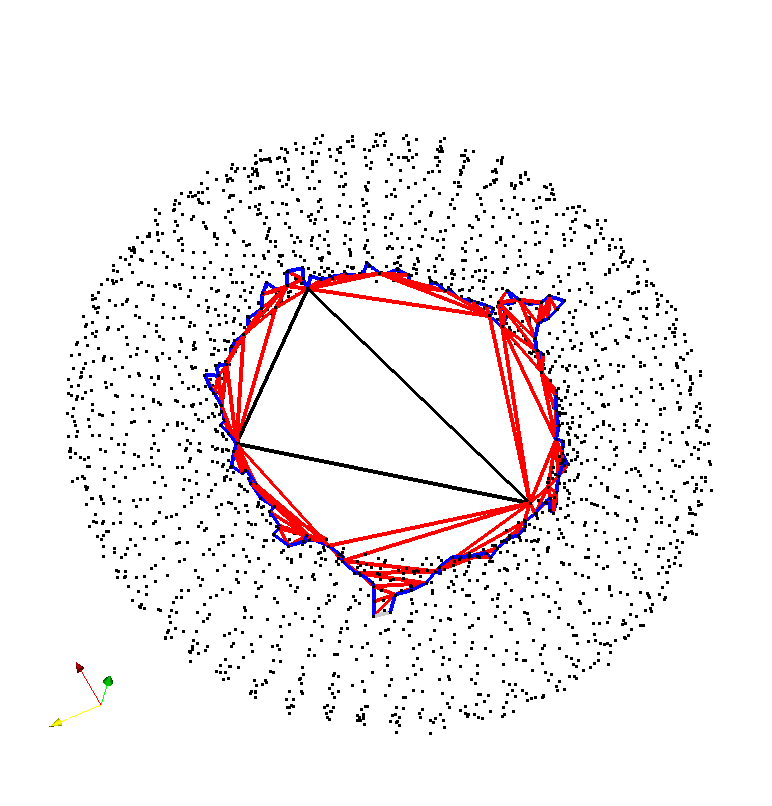}
  \includegraphics[width=0.3\hsize]{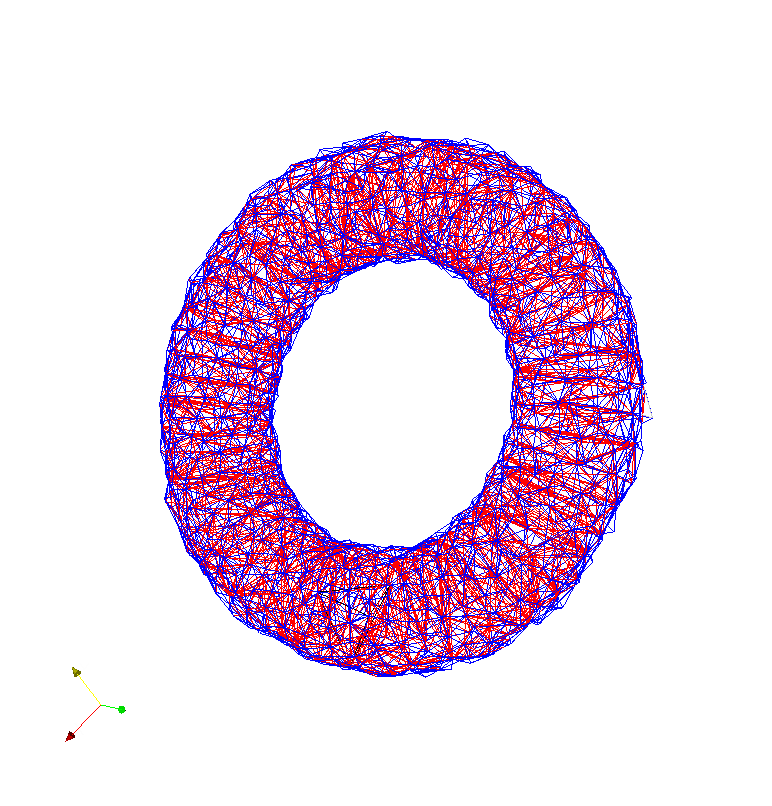}
  \caption{Volume optimal cycles for $(0.001, 0.072)$ and $(0.001, 0.453)$ in $D_1$ and
    $(0.008, 0.081)$ in $D_2$ 
    on the torus point cloud.}
  \label{fig:torus-voc}
\end{figure}

Figure~\ref{fig:torus-voc} shows the volume optimal cycles of these three birth-death pairs
using Algorithm~\ref{alg:volopt}.
Blue lines show volume optimal cycles, red lines show optimal volumes,
black lines show $\sigma_d$ for each birth death pair $(b, d)$ (we call this simplex the \textit{death simplex}). Black dots show the point cloud. By the figure, we understand
how homology generators appear and disappear in the filtration of the
torus point cloud.
The computation times are 25sec, 33sec, and 7sec on our laptop PC.

By using Algorithm~\ref{alg:volopt-hd-compute}, we can also compute volume optimal cycles
in $D_2$. In this example, the computation time by
Algorithm~\ref{alg:volopt-hd-compute} is about 2sec. This is much faster than
Algorithm~\ref{alg:volopt} even if Algorithm~\ref{alg:volopt-hd-compute} computes
\emph{all} volume optimal cycles.

\subsection{Amorphous silica}
\label{sec:example-silica}

In this example, we use the atomic configuration of 
amorphous silica computed by molecular dynamical simulation
as a point cloud and we try to reproduce the result
in \cite{Hiraoka28062016}. In this example, we use weighted alpha
filtration whose weights are the radii of atoms. The number of atoms are
8100, 2700 silicon atoms and 5400 oxygen atoms.

Figure~\ref{fig:amorphous-silica} shows the 1st PD. This diagram
have four characteristic areas $C_P$, $C_T$, $C_O$, and $B_O$.
These areas correspond to
the typical ring structures in the amorphous silica as follows.
Amorphous silica consists of silicon atoms and oxygen atoms and
the network structure is build by covalent bonds between silicons and oxygens.
$C_P$ has rings whose atoms are \ce{$\cdots$ -Si-O-Si-O- $\cdots$ } where
\ce{-} is a covalent bond between a silicon atom and a oxygen atom.
$C_P$ has triangles consisting of \ce{O-Si-O}. 
$C_O$ has triangles consisting of three oxygen atoms appearing alternately
in \ce{$\cdots$-O-Si-O-Si-O-$\cdots$}. 
$B_O$ has many types of ring structures, but one typical ring is a quadrangle 
consists of four oxygen atoms 
appearing alternately
in \ce{$\cdots$-O-Si-O-Si-O-Si-O-$\cdots$}.

Figure~\ref{fig:voc-silica} shows the volume optimal cycles for birth-death pairs
in $C_P, C_t, C_O$ and $B_O$. In this figure oxygen (red) and silicon (blue) atoms
are also shown in addition to volume optimal cycles, optimal volumes,
and death simplices.
We can reproduce the result of \cite{Hiraoka28062016} about ring reconstruction.

\begin{figure}[tbp]
  \centering
  \includegraphics[width=0.24\hsize]{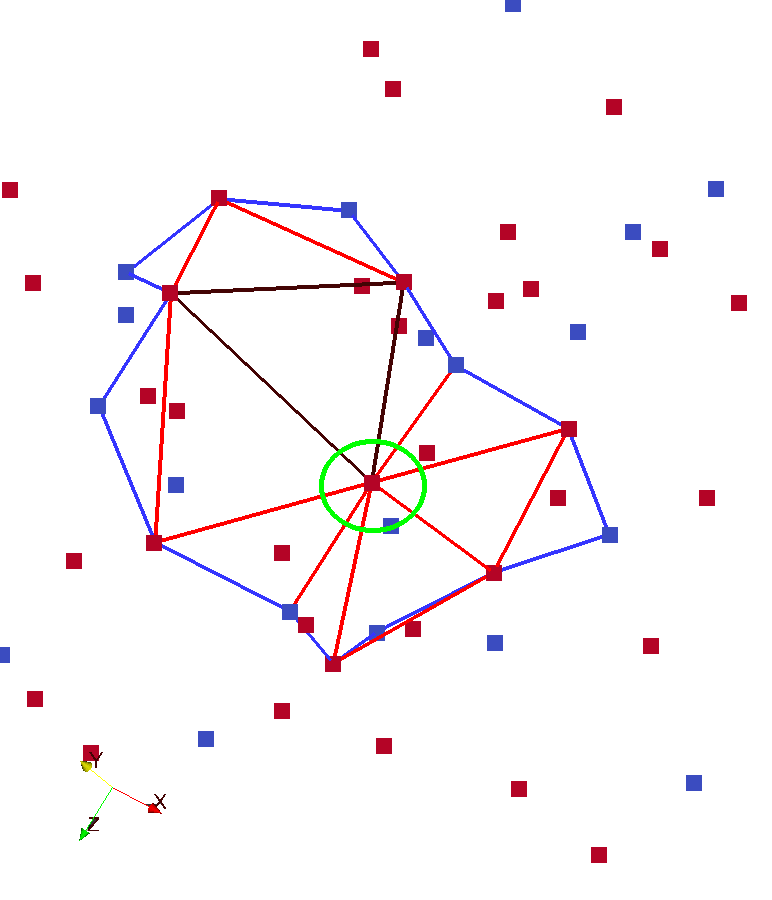}
  \includegraphics[width=0.24\hsize]{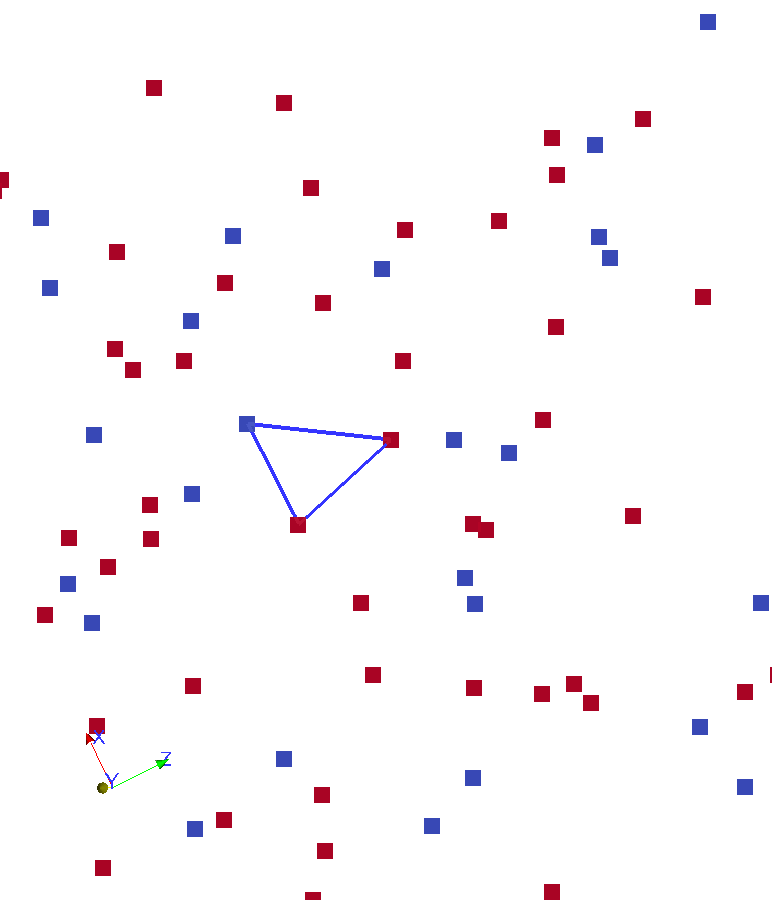}
  \includegraphics[width=0.24\hsize]{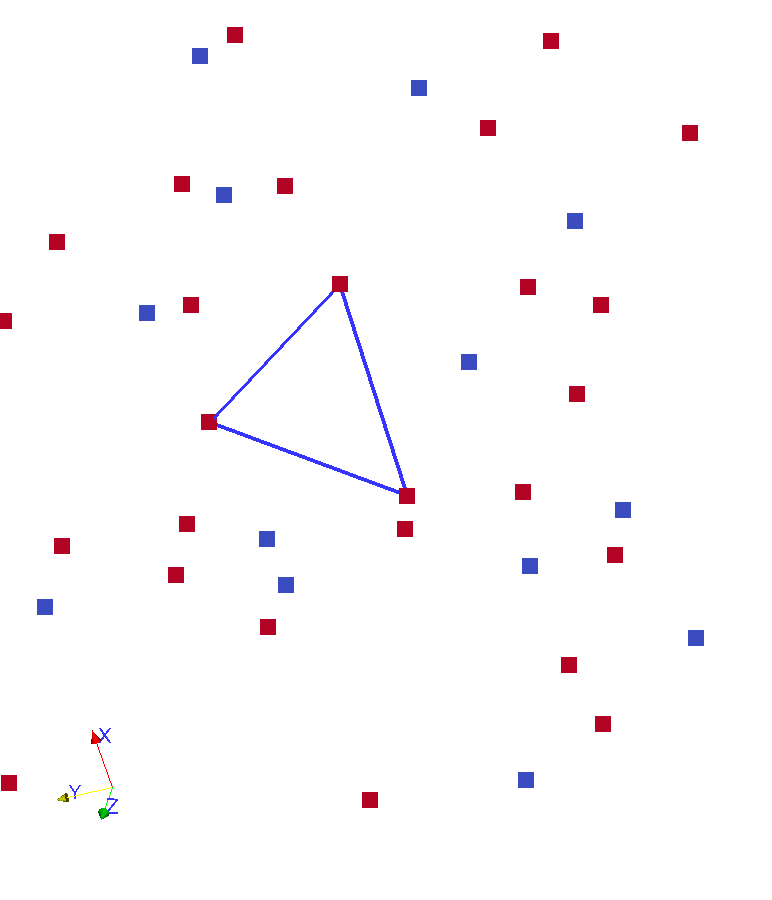}
  \includegraphics[width=0.24\hsize]{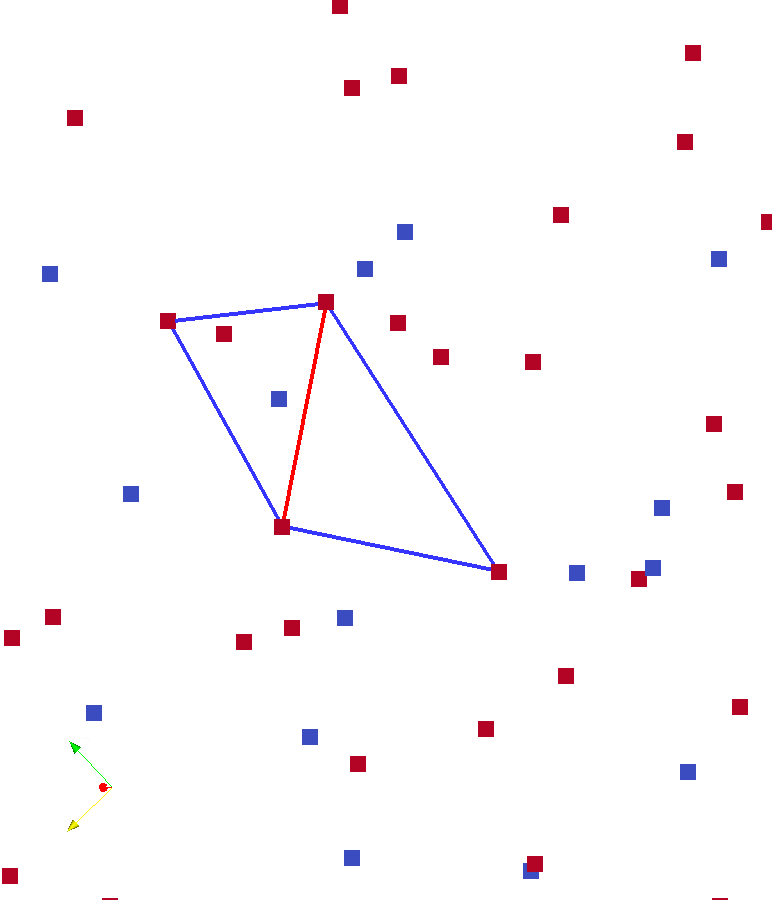}
  \caption{Volume optimal cycles in amorphous silica in $C_P, C_T, C_O$, and $B_O$ (from left to right).}
  \label{fig:voc-silica}
\end{figure}

We also know that the oxygen atom rounded by the green circle in this figure
is important to determine the death time. The death time of this birth-death pair is
determined by the radius of circumcircle of the black triangle (the death simplex),
hence if the oxygen atom moves away, the death time becomes larger.
The oxygen atom is contained in another \ce{$\cdots$ -Si-O-Si-O- $\cdots$}
ring structure around the volume optimal cycle (the blue ring). By the
analysis of the optimal volume, we clarify that such an interaction of covalent bond
rings determines the death times of birth-death pairs
in $C_P$. This analysis is impossible for the optimal cycles, and
the volume optimal cycles enable us to analyze persistence diagrams more deeply.

\begin{figure}[tbp]
  \centering
  \includegraphics[width=0.5\hsize]{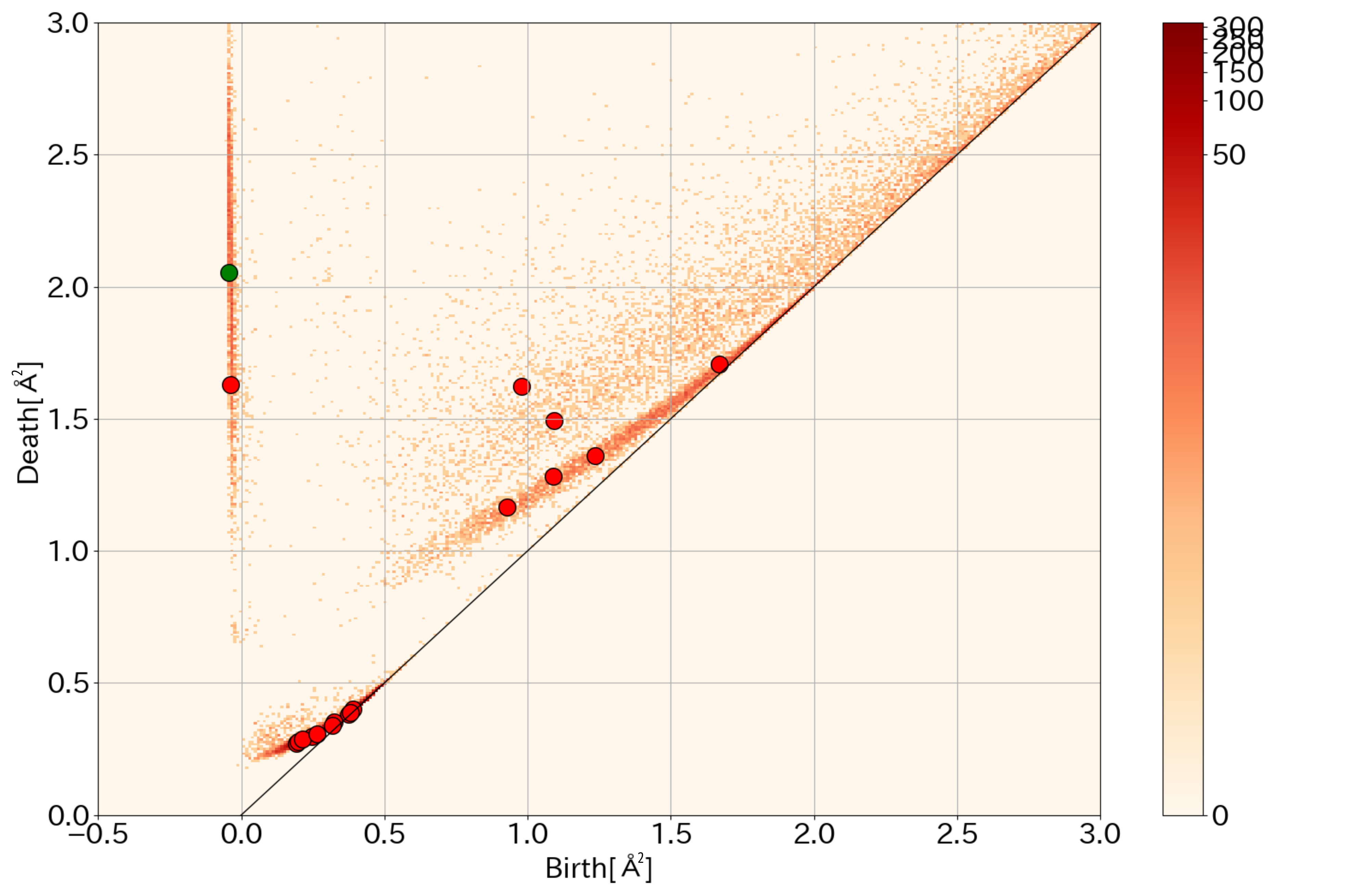}
  \caption{Children birth-death pairs. 
    Red circles are children birth-death pairs of the green birth-death pair.}
  \label{fig:children-bd-pairs}
\end{figure}

Figure~\ref{fig:children-bd-pairs} shows the children birth-death pairs of
the green birth-death pair. The rings corresponding to these children birth-death
pairs are subrings of the large ring corresponding to the green birth-death pair.
This computation result shows that a ring in $C_P$ has subrings in $C_T$, $C_O$,
and $B_O$. The hierarchical structure of these rings
shown in \cite{Hiraoka28062016}. We can easily find such a hierarchical structure
by using our new algorithm.

The computation time is 3 or 4 seconds for each
volume optimal cycle on the laptop PC. The computation time for amorphous silica
is much less than
that for 2-torus even if the number of points in amorphous silica is larger than
that in 2-torus. This is because the locality of
volume optimal cycles works very fine in the  example of amorphous silica.

\subsection{Face centered cubic lattice with defects}

The last example uses the point cloud of face centered cubic (FCC) lattice
with defects. By this example, we show how to use the persistence trees
computed by Algorithm~\ref{alg:volopt-hd-compute}.
The point cloud is prepared by constructing perfect FCC lattice,
adding small Gaussian noise to each point,
and randomly removing points from the point cloud. 

\begin{figure}[thbp]
  \centering
  \includegraphics[width=0.8\hsize]{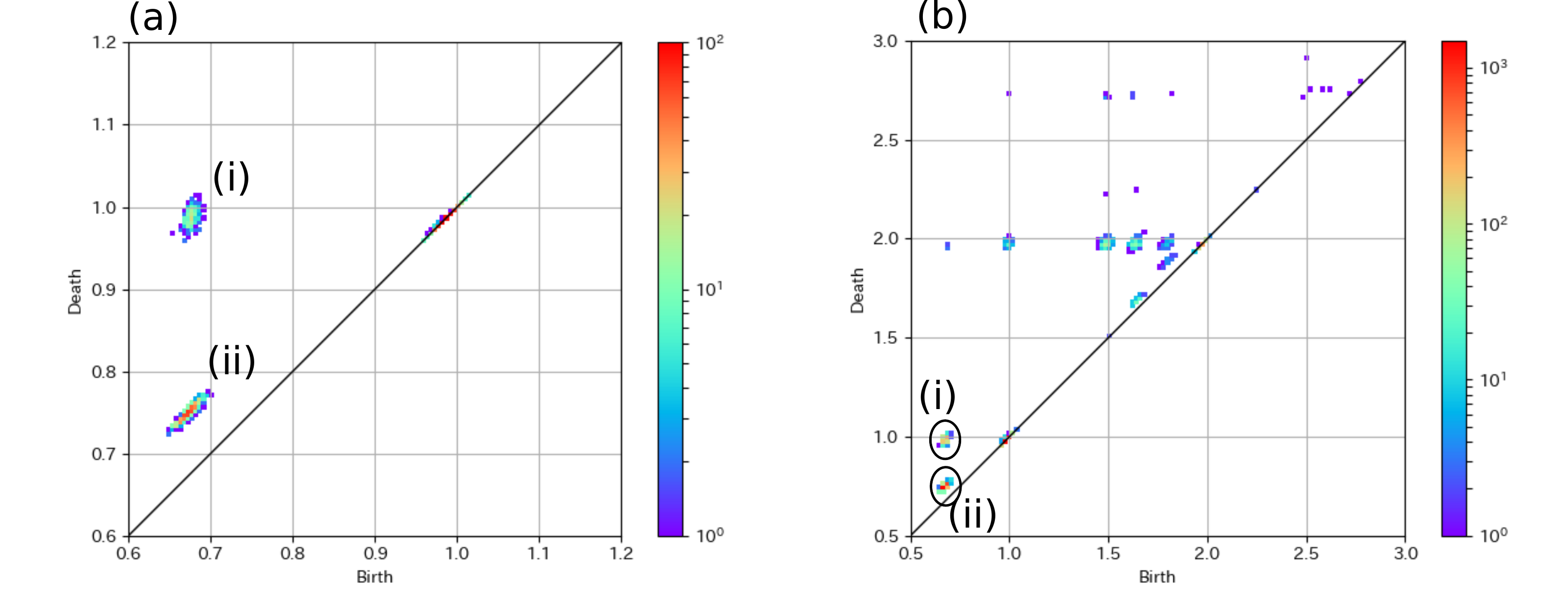}
  \caption{(a) The 2nd PD of the perfect FCC lattice with small Gaussian noise.
    (b) The 2nd PD of the lattice with defects.
  }\label{fig:fcc-pd}
\end{figure}

Figure~\ref{fig:fcc-pd}(a) shows the 2nd PD of FCC lattice with small
Gaussian noise. (i) and (ii) in the figure correspond to
octahedron and tetrahedron cavities in the FCC lattice.
In materials science, these cavities are famous as octahedron sites and tetrahedron sites.
Figure~\ref{fig:fcc-pd}(b) shows the 2nd PD of the lattice with defects.
In the PD, birth-death pairs corresponding to
octahedron and tetrahedron cavities remain ((i) and (ii) in Fig~\ref{fig:fcc-pd}(b)),
but other types of birth-death pairs appear in this PD. These pairs
correspond to other types of cavities generated by removing points from the FCC lattice.

Figure~\ref{fig:fcc-tree-1}(a) shows a tree computed
by Algorithm~\ref{alg:volopt-hd-compute}. Red markers are nodes of the tree,
and lines between two markers are edges of the tree, where
upper left nodes are ancestors and lower right nodes are descendants.
The tree means that the largest cavity corresponding to most upper-left node
has sub cavities corresponding descendant nodes.
Figure~\ref{fig:fcc-tree-1}(b) shows the volume optimal cycle of
the most upper-left node, (c) shows the volume optimal cycles of pairs
in (i), and (d) shows the volume optimal cycles of pairs in (ii).
Using the algorithm,
we can study the hierarchical structures of the 2nd PH.

\begin{figure}[thbp]
  \centering
  \includegraphics[width=0.85\hsize]{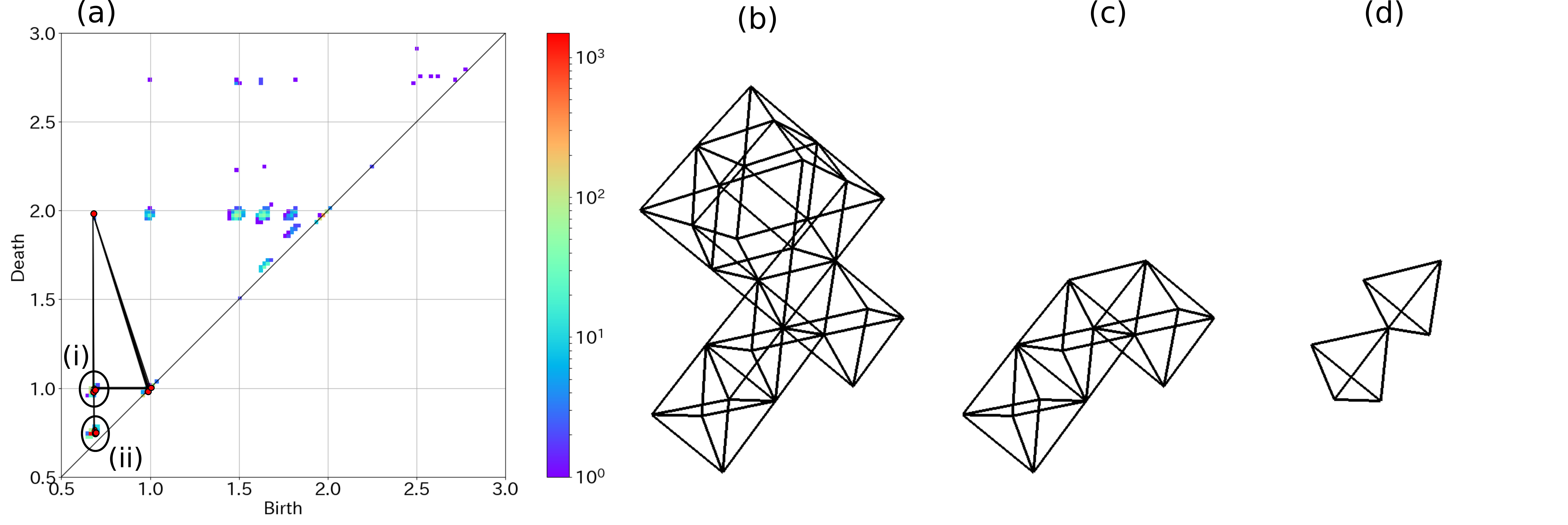}
  \caption{A persistence tree and related volume optimal cycles.
    (a) The persistence tree whose root is $(0.68, 1.98)$.
    (b) The volume optimal cycle of the root pair.
    (c) The volume optimal cycles of birth-death pairs in (i) which are descendants of
    the root pair.
    (d) The volume optimal cycles of birth-death pairs in (ii) which are descendants of
    the root pair.
  }
  \label{fig:fcc-tree-1}
\end{figure}

\subsection{Computation performance comparison with optimal cycles}
\label{subsec:performance}
We compare the computation performance between optimal cycles and volume optimal cycles.
We use OptiPers for the computation of optimal cycles for persistent homology,
which is provided by Dr. Escolar, one of the authors of \cite{Escolar2016}.
OptiPers is written in C++ and our software is mainly written in python,
and python is much slower than C++, so
the comparison is not fair, but suggestive for the readers.

We use two test data.
One test data is the atomic configuration of amorphous silica used in the above example.
The number of points is 8100.
Another data is the partial point cloud of the amorphous silica.
The number of points is 881. We call these data the large data and the small data.
Table~\ref{tab:performance} shows the computation time of 
optimal cycles/volume optimal cycles for all birth-death pairs in the 1st PD
by OptiPers/Homcloud. 

\begin{table}[thbp]
  \centering
  \begin{tabular}{c|cc}
    & optimal cycles (OptiPers) & volume optimal cycles (Homcloud) \\ \hline
    the small data & 1min 17sec & 3min 9sec  \\
    the large data & 5hour 46min & 4hour 13min\\
  \end{tabular}
  \caption{Computation time of optimal cycles and volume optimal cycles on the large/small data.}
  \label{tab:performance}
\end{table}

For the small data, OptiPers is faster than Homcloud, but to the contrary,
for the large data, Homcloud is faster than OptiPers. This is because
the performance improvement technique using the locality of the optimal volume
works fine for the large data, but for the small data the technique is not
so effective and the overhead cost using python is dominant for Homcloud.
This benchmark shows that the volume optimal cycles have an advantage about
the computation time when an input point cloud is large.

\section{Conclusion}\label{sec:conclusion}

In this paper, we propose the idea of volume optimal cycles to identify
good geometric realizations of homology generators appearing in persistent homology.
Optimal cycles are proposed for that purpose in \cite{Escolar2016},
but our method is faster for large data and
gives better information. Especially, we can reasonably compute 
children birth-death pairs
only from a volume optimal cycle. Volume optimal cycles
are already proposed under the limitation of dimension in \cite{voc},
and this paper generalize the idea.

Our idea and algorithm are widely applicable to and
useful for
the analysis of point clouds in $\R^n$ by using the (weighted) alpha filtrations.
Our method gives us intuitive understanding of PDs. In~\cite{PDML}, such inverse analysis
from a PD to its original data is effectively used to study many geometric data
with machine learning on PDs and our method is useful to
the combination of persistent homology and machine learning.

In this paper, we only treat simplicial complex, but our method is
also applicable to a cell filtration and a cubical filtration.
Our algorithms will be useful to study sublevel or superlevel filtrations
given by 2D/3D digital images.

\appendix

\section{Proofs of Section \ref{sec:vochd}}\label{sec:pfvochd}

The theorems shown in this section are a kind of folklore theorems.
Some researchers about persistent homology probably know the fact that the merge-tree
algorithm gives a 0th PD, and the algorithm
is available to compute an $(n-1)$-th PD using Alexander duality,
but we cannot find
the literature for the complete proof.
\cite{voc} stated that the algorithm also gives the tree structure on an
$(n-1)$-th PD, but the thesis does not have the complete proof.
Therefore we will show the proofs here.

Alexander duality says that for any good topological subspace $X$ of $S^n$,
the $(k-1)$-th homology of $X$ and $(n-k)$-th cohomology of $S^n\backslash X$ have
the same information. In this section, we show Alexander duality theorem
on persistent homology.
In this section, we always use $\Zint_2$ as a coefficient of homology and cohomology.

\subsection{Persistent cohomology}
The persistent cohomology is defined on a decreasing sequence
$\Y : Y_0 \supset \cdots \supset Y_K$ of topological spaces.
The cohomology vector spaces and the linear maps induced from
inclusion maps define the sequence
\begin{align*}
  H^q(Y_0) \to \cdots \to H^q(Y_K),
\end{align*}
and this family of maps is called persistent cohomology $H^q(\mathbb{Y})$.
The decomposition theorem also holds for persistent cohomology in the same way as
persistent homology and we define the $q$th cohomologous persistence 
diagram $D^q(\mathbb{Y})$
using the decomposition.

\subsection{Alexander duality}\label{sec:alex}
Before explaining
Alexander duality, we show the following proposition about the dual decomposition.

\begin{prop}\label{prop:dual}
  For any oriented closed $n$-manifold $S$ and its simplicial decomposition $K$,
  there is a decomposition of $M$, $\bar{K}$, satisfying the followings:
  \begin{enumerate}
  \item $\bar{K}$ is a cell complex of $M$.
  \item There is a one-to-one correspondence between $K$ and $\bar{K}$.
    For $\sigma \in K$, we write the corresponding cell in $\bar{K}$ as $\bsigma$.
  \item $\dim \sigma = n - \dim \bsigma$ for any $\sigma \in K$.
  \item If $X \subset K$ a subcomplex of $K$, 
    \begin{align*}
      \bar{X} = \{\bsigma \mid \sigma \not\in X\}
    \end{align*}
    is a subcomplex of $\bar{K}$.
  \item We consider the chain complex of $K$ and $\bar{K}$,
    let $\partial$ and $\bar{\partial}$ be boundary operators on
    those chain complexes, and let $B$ and $\bar{B}$ be matrix representations of
    $\partial$ and $\bar{\partial}$, i.e.
    $\partial \sigma_i = \sum_j B_{ji} \sigma_j$ and
    $\bar{\partial} \bsigma_i = \sum_j \bar{B}_{ji} \bsigma_j$.
    Then $\bar{B}$ is the transpose of $B$.
  \end{enumerate}
\end{prop}

This decomposition $\bar{K}$ is called the \textit{dual decomposition} of $K$.
One example of the dual decomposition is a Voronoi decomposition 
with respect to a Delaunnay triangulation.
Using the dual decomposition, we can define the map $\theta$ from
$C_k(K)$ to $C^{n-k}(\bar{K})$ 
for $k=0,\ldots,n$
as the linear extension of $\sigma_i \mapsto \bsigma_i^*$, where
$\{\sigma_i\}_i$ are $k$-simplices of $K$,
$\{\bsigma_i\}_i$ are corresponding $(n-k)$-cells of $\bar{K}$,
and
$\{\bsigma_i^* \in C^{n-k}(\bar{K})\}_i$ is the dual basis of 
$\{\bsigma_i\}_i$.

The map $\theta$ satisfies the equation
\begin{align}
  \theta \circ \partial = \delta \circ \theta, \label{eq:comm_poincare}
\end{align}
where $\delta$ is the coboundary
operator on $C^*(\bar{K})$ from Proposition \ref{prop:dual}. The map $\theta$ induces the
isomorphism $H_k(K) \simeq H^{n-k}(\bar{K})$, and the isomorphism is called
Poincar\'{e} duality.

Using the dual decomposition, we show Alexander duality theorem.
\begin{theorem}
  For an $n$-sphere $S^n$, its simplicial decomposition $K$, and
  a subcomplex of $X \subset K$, we take a dual decomposition $\bar{K}$ and
  a subcomplex of $\bar{X}$ as in Proposition \ref{prop:dual}. Then,
  \begin{align}
    \tilde{H}_{k-1}(X) \simeq \tilde{H}^{n-k}(\bar{X}) \label{eq:alex_iso}
  \end{align}
  holds for any $k=1,\ldots,n$, where $\tilde{H}$ is the
  reduced (co)-homology.
\end{theorem}
To apply the duality theorem to persistent homology,
we investigate that isomorphism in detail.

First, we consider the case of $K = X$. In this case, $\bar{X} = \emptyset$
and the homology of $X$ is the same as an $n$-sphere.
Therefore, $\tilde{H}_k(X) = 0$ for any $k=0,\ldots, n-1$
and this is isomorphic to $\tilde{H}^{n-k-1}(\emptyset) = 0$. 

Next, we consider the case that $K \not = X$. In this case, there is a
$n$-simplex of $K$ which is not contained in $X$. We write the $n$-simplex
as $\omega$ and let $K_0$ be $K\backslash\{\omega\}$.

\begin{prop}
  There is the following isomorphism.
  \begin{align}
    H_k(K, X) \simeq H^{n-k}(\bar{X}).
    \label{eq:kx_barx_iso}
  \end{align}
  This isomorphism is induced by:
  \begin{align*}
    \bar{\theta} : C_k(K, X) = C_k(K)/C_k(X) &\to C^{n-k}(\bar{X}) \\
    \sum_{i=s+1}^t a_i \sigma_i + C_k(X) &\mapsto \sum_{i=s+1}^t a_i \bsigma_i^*
  \end{align*}
  where $\{\sigma_1,\ldots,\sigma_t\}$ is all $k$-simplices of $K$ and
  $\{\sigma_1,\ldots,\sigma_s\}$ is all $k$-simplices of $X$.
\end{prop}
The map $\theta$ is well-define and isomorphic since
$\{\bsigma_{s+1}, \ldots, \bsigma_{t}\}$ is equal to the set of
all $(n-k)$-simplices of $\bar{X}$. In addition, 
$\delta \circ \bar{\theta} = \bar{\theta} \circ \partial$ holds
where $\partial$ is the boundary operator on $C_*(K, X)$, and
$\delta$ is the coboundary operator on $C^*(\bar{X})$
due to \eqref{eq:comm_poincare}.
Using the map $\bar{\theta}$, the isomorphism 
$\bar{\theta}_* : H_k(K, X) \to H^{n-k}(\bar{X})$ is defined as follows,
\begin{align}
  \left[\sum_{i=s+1}^t a_i \sigma_i + C_k(X)\right] \mapsto 
  \left[\sum_{i=s+1}^t a_i \bsigma_i^*\right]. \label{eq:bar_theta_star}
\end{align}

The next key is the long exact sequence on the pair $(K, X)$. 
\begin{align}
  \cdots \to
  \tilde{H}_k(X) \to
  \tilde{H}_k(K) \xrightarrow{j_*}
  H_k(K, X) \xrightarrow{\partial_*}
  \tilde{H}_{k-1}(X) \to \tilde{H}_{k-1}(K) \to \cdots . \label{eq:long_exact_seq}
\end{align}
The map $\partial_*$ is written as follows
\begin{align}
  \partial_*([z + C_k(X)]) = [\partial z],
  \label{eq:partial_star}
\end{align}
and $j_*$ is induced by the projection map from $C_k(K)$ to $C_k(K, X)$.
If $k \not = n$, since both $\tilde{H}_k(K)$ and $\tilde{H}_{k-1}(K)$ are zero,
The following map is isomorphic due to the long exact sequence \eqref{eq:long_exact_seq}.
\begin{align}
  \partial_* : H_{k}(K, X) \xrightarrow{\sim} H_{k-1}(K). \label{eq:partial_star_iso}
\end{align}
By combining \eqref{eq:kx_barx_iso} and \eqref{eq:partial_star_iso},
we conclude the isomorphism \eqref{eq:kx_barx_iso} for $k \not = n$.
We can explicitly write the isomorphism from $\tilde{H}^{n-k}$
to $\tilde{H}_{k-1}$ as follows using
\eqref{eq:bar_theta_star} and \eqref{eq:partial_star}:
\begin{align}
  \left[\sum_{i=s+1}^t a_i \bsigma_i^*\right] \mapsto 
  \left[\partial\left(\sum_{i=s+1}^t a_i \bsigma_i\right)\right].
\end{align}

When $k=n$, we need to treat the problem more carefully.
From the long exact sequence \eqref{eq:long_exact_seq}, we can show that
the following sequence is exact:
\begin{align}
  \begin{array}{ccccccccc}
    H_n(X)&\to&H_n(K)&\xrightarrow{j_*}&H_n(K, X)&\xrightarrow{\partial_*}&H_{n-1}(X)&\to&H_{n-1}(K) \\
    \veq && \vsimeq & & & & & & \veq \\
    0 && \Zint_2    & & & & & & 0 \\
  \end{array}.
\end{align}

Let $\{\sigma_1, \ldots, \sigma_{t-1}, \sigma_t=\omega\}$ be
$n$-simplices of $K$ and $\{\sigma_1, \ldots, \sigma_s\}$ be
$n$-simplices of $X$. From the assumption of $X \not = K$,
$s < t$ holds.
It is easy to show that $\tau = \sigma_1 + \cdots + \sigma_t$ is
the generator of $Z_n(K)$. From the definition of
the reduced cohomology, 
\begin{align}
 \tilde{H}^0(\bar{X}) = Z^0(\bar{X})/\left<\bar{\tau}\right>,  \label{eq:cohomology0}
\end{align}
 where
$Z^0(\bar{X}) = \ker(\delta: C^0(\bar{X}) \to C^1(\bar{X}))$ and
$\bar{\tau} = \bar{\theta}(j(\tau)) = \bsigma_{s+1}^* + \cdots + \bsigma_t^*$.
For $\Zint_2$ coefficient, the following set is the basis of $Z^0(\bar{X})$
\begin{align}
  \{ \sum_{\bsigma \in C}\bsigma^* \mid C \in \textrm{cc}(\bar{X}) \},
  \label{eq:z0_basis}
\end{align}
where $\textrm{cc}(\bar{X})$ is the connected component decomposition of 0-cells
in $\bar{X}$. Therefore, we can write $\tilde{H}^0(\bar{X})$ as:
\begin{align*}
  \tilde{H}^0(X) = \{ [\sum_{\bsigma \in C}\bsigma^*]
  \mid C \in \textrm{cc}_\omega(\bar{X}) \},
\end{align*}
where $\textrm{cc}_\omega(\bar{X}) \subset \textrm{cc}(\bar{X})$ 
is the set of connected components which do not contain $\bar{\omega}$.
Using the above relations, we can show $\tilde{H}^0(\bar{X}) \simeq H_{n-1}(X)$
whose isomorphism is the linear extension of the following:
\begin{equation}
  \begin{aligned}
    \Theta&: \tilde{H}^0(X) \to H_{n-1}(X) \\ 
    \Theta&([\sum_{\bsigma \in C}\bsigma^*]) =
            [\partial(\sum_{\bsigma \in C}\sigma)]   \\
          &\textrm{for all } C \in \textrm{cc}_\omega(\bar{X}). 
  \end{aligned}
  \label{eq:ccboundary}
\end{equation}

\subsection{Alexander duality and persistent homology}
\label{subsec:alex_ph}

To apply Alexander duality to the persistent homology, we need to
consider the relation between inclusion maps and the isomorphism $\bar{\theta}_*$.
For two subcomplex $X_1 \subset X_2$ of $K$, the following diagram commutes:
\begin{equation}
  \label{eq:alexph_comm}
  \begin{aligned}
    \begin{CD}
      C_\ell(K, X_1) @>\bar{\theta}>> C^{n-\ell}(\bar{X}_1) \\
      @VV{\phi}V                           @VV\bar{\phi}^{\vee}V \\
      C_\ell(K, X_2) @>\bar{\theta}>> C^{n-\ell}(\bar{X}_2), \\
    \end{CD}
  \end{aligned}
\end{equation}
where $\phi$ and $\bar{\phi}^\vee$ are induced from the inclusion maps.
Note that
$X_1 \subset X_2$ induces $\bar{X}_1 \supset \bar{X}_2$ and
$\bar{\phi}^\vee$ is defined from
$C^{n-\ell}(\bar{X}_1)$ to $C^{n-\ell}(\bar{X}_2)$.
Using \eqref{eq:alexph_comm},
we have the following commutative diagram:
\begin{equation}
  \label{eq:alexph_comm2}
  \begin{aligned}
    \begin{CD}
      H_\ell(K, X_1) @>\bar{\theta_*}>> H^{n-\ell}(\bar{X}_1) \\
      @VV{\phi_*}V                           @VV\bar{\phi}^{*}V \\
      H_\ell(K, X_2) @>\bar{\theta_*}>> H^{n-\ell}(\bar{X}_2), \\
    \end{CD}
  \end{aligned}
\end{equation}
We also have the following commutative diagram between two long exact sequences
\begin{equation}
  \label{eq:alexph_comm3}
  \begin{aligned}
    \begin{CD}
      \cdots  @>>> \tilde{H}_k(X_1) @>>> \tilde{H}_k(K) @>j_*>>
      H_{k}(K, X_1) @>\partial_*>>
      \tilde{H}_{k-1}(X_1) @>>>  \cdots \\
      @. @VV{\phi}_*V @| @VV\phi_*V @VV\phi_*V  @. \\
      \cdots  @>>> \tilde{H}_k(X_2) @>>> \tilde{H}_k(K) @>j_*>>
      H_{k}(K, X_2) @>\partial_*>>
      \tilde{H}_{k-1}(X_2) @>>>  \cdots . \\
    \end{CD}
  \end{aligned}
\end{equation}

From \eqref{eq:alexph_comm2}, \eqref{eq:alexph_comm3}, and the discussion
in Section~\ref{sec:alex}, we have the following commutative diagram:
\begin{align*}
  \begin{CD}
    \tilde{H}_{\ell-1}(X_2) @>\sim>> \tilde{H}^{n-\ell}(\bar{X}_2) \\
    @VV\phi_*V                            @VV\bar{\phi}^*V \\
    \tilde{H}_{\ell-1}(X_1) @>\sim>> \tilde{H}^{n-\ell}(\bar{X}_1).
  \end{CD}
\end{align*}
This diagram means that the isomorphism preserves
the decomposition structure of persistent homology and hence
$\tilde{H}_{\ell-1}(\X) \simeq \tilde{H}^{n-\ell}(\bar{\X})$ holds
for $\X: X_0 \subset \cdots \subset X_K$
where $\bar{\X} : \bar{X}_0 \supset \cdots \supset \bar{X}_K$.
\subsection{Alexander duality and a triangulation in $\R^n$}
\label{sec:alex_alpha}

Here, we consider the simplicial filtration in $\R^n$ satisfying
Condition~\ref{cond:rn}. Under the condition,
we need to embed the filtration $\X$
on $\R^n$ into $S^n$ by using one point compactification. We consider a
embedding $|X| \to S^n$ and take $\sigma_{\infty}$ as $S^n\backslash |X|$.
Using the embedding,
we can regard $X \cup \{\sigma_\infty\}$ as a cell decomposition of $S^n$.
The above discussion about Alexander duality on persistent homology
works on this cell complex,
if we properly define the boundary operator and the dual decomposition.
In that case, we regard $\sigma_\infty$ as $\omega$ in the definition of $K_0$.

\subsection{Merge-Tree Algorithm for 0th Persistent Cohomology}
\label{subsec:treemerge-0-pcohom}
The above discussion shows that
only we need to do is to give an algorithm for
computing $0$th persistent cohomology of the dual filtration
In fact, we can efficiently compute 
the $0$th cohomologous persistence diagram
using the following merge-tree algorithm. 

To simplify the explanation of the algorithm, we assume the following condition.
This condition corresponds to Condition~\ref{cond:ph} for persistent homology.
\begin{cond}\label{cond:cohom}\ 
  \begin{itemize}
  \item
    $Y = \{\bsigma_1, \ldots \bsigma_K, \bsigma_\infty \}$ is a cell complex and
    $Y_k = \{\bsigma_{k+1},\ldots, \bsigma_K, \bsigma_\infty\}$ is a subcomplex of $Y$ for any
    $0 \leq k < K$.
  \item $\bsigma_{\infty}$ is 0-cell and $Y_{K} = \{\bsigma_{\infty}\}$ is also a
    subcomplex of $Y$.
  \item $Y$ is connected.
  \end{itemize}
\end{cond}

Under the condition, we explain the algorithm to compute
the decomposition of 0th persistence cohomology on the
decreasing filtration
$\mathbb{Y}:Y = Y_0 \supset \cdots \supset Y_{K} = \{\bsigma_\infty\}$. 

Algorithm~\ref{alg:tree-0-compute} computes the 0th cohomologous persistence diagram.
In this algorithm, $(V_k, E_k)$ is a graph whose nodes are 0-cells of $Y$ and
whose edges have extra data in $\Zint$.
Later we show this algorithm is applicable for
computing $D_{n-1}(\X)$ using Alexander duality.

\begin{algorithm}[h!]
  \caption{Merge-Tree algorithm for the 0th cohomologous PD}\label{alg:tree-0-compute}
  \begin{algorithmic}
    \Procedure{Compute-Tree}{$\mathbb{Y}$}
    \State initialize $V_{K} = \{\bsigma_\infty\}$ and $E_{K} = \emptyset$
    \For{$k=K,\ldots,1$}
      \If{$\bsigma_k$ is a $0$-simplex}
        \State $V_{k-1} \gets V_{k} \cup \{\bsigma_k\},\ E_{k-1} \gets E_{k}$
      \ElsIf{$\bsigma_k$ i a $1$-simplex}
        \State let $\bsigma_s, \bsigma_t$ are two endpoints of $\bsigma_k$
        \State $\bsigma_{s'} \gets \textproc{Root}(\bsigma_s, V_{k}, E_{k})$
        \State $\bsigma_{t'} \gets \textproc{Root}(\bsigma_t, V_{k}, E_{k})$
        \If{$s'=t'$}
          \State $V_{k-1} \gets V_{k},\ E_{k-1} \gets E_{k}$
        \ElsIf{$s'> t'$}
        \State $V_{k-1} \gets V_{k},\ 
        E_{k-1} \gets E_{k}\cup \{(\bsigma_{t'} \xrightarrow{k} \bsigma_{s'})\}$
        \Else
        \State $V_{k-1} \gets V_{k},\ 
        E_{k-1} \gets E_{k}\cup \{(\bsigma_{s'} \xrightarrow{k} \bsigma_{t'})\}$
        \EndIf
      \Else
        \State $V_{k-1} \gets V_{k},\  E_{k-1} \gets E_{k}$
      \EndIf
    \EndFor
    \Return $(V_0, E_0)$
    \EndProcedure
  \end{algorithmic}
\end{algorithm}

This algorithm tracks all $\{(V_k, E_k)\}_{k=0,\ldots,K}$ for the mathematical
proof, but when you implement the algorithm, you do not need to keep the history
and you can directly update the set of nodes and edges.

\begin{theorem}\label{thm:tree0}
  The 0th reduced cohomologous persistence diagram $\tilde{D}^0(\mathbb{Y})$
  is given as follows:
  \begin{align*}
    \tilde{D}^0(\mathbb{Y}) = \{(k, s) \mid (\bsigma_s \xrightarrow{k} \bsigma_t) \in E_0\}
  \end{align*}
\end{theorem}

To prove the theorem and justify the algorithm, we show some basic facts
about the graph $(V_K, E_k)$ given by the algorithm.
These facts are shown by checking
the edges/nodes adding rule of each step in Algorithm~\ref{alg:tree-0-compute}.

\begin{fact}\label{fact:vk}
  $V_k = \{\bsigma_\ell : \text{0-simplex in } Y \mid k < \ell\}$
\end{fact}
Fact~\ref{fact:vk} is obvious from the algorithm.

\begin{fact}\label{fact:graph-is-tree}
  For any $k$, $(V_k, E_k)$ is a forest, i.e. a set of trees. 
  That is, the followings hold:
  \begin{itemize}
  \item There is no loop in the graph
  \item For any node, the number of outgoing edges from the node is zero or one.
    \begin{itemize}
    \item If the number is zero, the node is a root node
    \item If the number is one, the node is a child node
    \end{itemize}
  \end{itemize}
\end{fact}
We can inductively prove Fact~\ref{fact:graph-is-tree}
since an edge is added between two roots of $(V_{k}, E_{k})$ in the algorithm.

\begin{fact}\label{fact:graph-cc}
  The topological connectivity of $Y_k$ is the same as
  $(V_k, T_k)$. That is,
  $\{\bsigma_{i_1}, \ldots, \bsigma_{i_\ell}\}$ is all 0-simplices
  of a connected component in $X_k$ if and only if
  there is a tree in $(V_k, E_k)$ whose nodes are
  $\{\bsigma_{i_1}, \ldots, \bsigma_{i_\ell}\}$.
\end{fact}
This is because the addition of a node to the graph corresponds 
to the addition of a connected component in $\mathbb{Y}$ and
the addition of an edge corresponds to the concatenation
of two connected components.

\begin{fact}\label{fact:tree_order}
  If there is a path
  $\bsigma_s \xrightarrow{k} \bsigma_t \to \cdots \to \bsigma_{s'} \xrightarrow{k'} \bsigma_{t'}$ in $(V_{k''}, E_{k''})$, the following inequality holds:
  \begin{align*}
    k' < k < s < s'.
  \end{align*}
\end{fact}

\begin{proof}[Proof of Fact~\ref{fact:tree_order}]
  The edge $\bsigma_s \xrightarrow{k} \bsigma_t$
  is added after $\bsigma_{s'} \xrightarrow{k'} \bsigma_{t'}$ is added
  in the algorithm since any edge is added between two root nodes, hence
  we have $k'<k$. We also show that $k < s < t$ and $k'<s'<t'$ from the
  rule of edge addition and this inequalities hold for any intermediate edge
  in the path, so we have $s < s'$. The required inequality comes from
  these inequalities.
\end{proof}

The following fact is shown since in the algorithm
each edge is added between two root nodes.
\begin{fact}\label{fact:subtree}
  If $\bsigma_s$ is not a root of a tree in $(V_k, E_k)$, the subtree
  whose root node is $\bsigma_s$ does not change in the sequence of graphs:
  $(V_{k}, E_{k}) \subset \cdots \subset (V_0, E_0)$.
\end{fact}

Using these facts, we set up the 0th persistence cohomology.
We prepare some symbols:
\begin{align*}
  R_k &= \{\bsigma_s \mid \bsigma_s \text{ is a root of a tree in } (V_k, E_k)\}, \\
  \desc_k(\bsigma_s) &= \{\bsigma_t : \mbox{a descendant node of $\bsigma_s$
                      in } (V_k, E_k) \mbox{, including $\bsigma_s$ itself}\}, \\
  \odesc_k(\bsigma_s) &= \{ \bsigma_t \in \desc_0(\bsigma_s) \mid k \leq t \}, \\
  y_s^{(k)} &= \sum_{\bsigma_t \in \desc_k(\bsigma_s)} \bsigma_t^* \in C^0(Y_k),\\
  \hat{y}_s^{(k)} &= \sum_{\bsigma_t \in \odesc_k(\bsigma_s)} \bsigma_t^* \in C^0(Y_k),\\
  \bar{\varphi}_k^\vee&: C^0(Y_k) \to C^0(Y_{k+1}) \ : \mbox{the induced map of the inclusion map $Y_k \xhookleftarrow{} Y_{k+1}$}.
\end{align*}

We prove the following lemma.
\begin{lem}\label{lem:pcohom-basis}
  $\{ \hat{y}_s^{(k)} \mid \bsigma_s \in R_k\}$ is a basis of $Z^0(Y_k) = H^0(Y_k)$.
\end{lem}
\begin{proof}
   From Fact~\ref{fact:graph-cc} and the theory of 0th cohomology, we have that
  $\{ y_s^{(k)} \mid \bsigma_s \in R_k\}$ are basis of $H^0(Y_k)$. Here, we prove the
  following three facts. Then the theory of linear algebra leads
  the statement of the lemma.
  \begin{enumerate}[(i)]
  \item $\#\{ y_s^{(k)} \mid \bsigma_s \in R_k\} = \# \{ \hat{y}_s^{(k)} \mid \bsigma_s \in R_k\} = \#R_k $
  \item Any $\hat{y}_s^{(k)}$ for $\bsigma_s \in R_k$ is a linear
    sum of $\{ y_s^{(k)} \mid \bsigma_s \in R_k\}$
  \item $\{ \hat{y}_s^{(k)} \mid \bsigma_s \in R_k\}$ are linearly independent.
  \end{enumerate}
  (i) is trivial. We show (ii). We can write
  $\hat{y}_s^{(k)}$ explicitly by using the two graphs $(V_{k}, E_k)$ and
  $(V_0, E_0)$ by the following way.
  Let $R_k(\bsigma_s)$ be
  \begin{align*}
    R_k(\bsigma_s) = \{\bsigma_t \in R_k \mid \bsigma_t \mbox{ is a descendant of }
    \bsigma_s \mbox{ in } (V_0, E_0), \mbox{ including $\bsigma_s$ itself} \}.
  \end{align*}
  Then we write $\hat{y}_s^{(k)} = \sum_{\bsigma_t \in R_k(\bsigma_s)} x_t^{(k)}$.
 We can show the equation from the followings.
 \begin{itemize}
 \item The family $\{ \desc_k(\bsigma_s) \mid \bsigma_s \in R_k \}$ 
   is pairwise disjoint.
 \item $\odesc_k(\bsigma_s) = \bigsqcup_{\bsigma_t \in R_k(\bsigma_t)} \desc_k(\bsigma_t)$
 \end{itemize}
 The first one comes from Fact~\ref{fact:graph-cc}. Next 
 $\odesc(\bsigma_s) \supset \bigsqcup_{\bsigma_t \in R_k(\bsigma_t)} \desc_k(\bsigma_t)$
 is shown.
 Pick any $\bsigma_u \in \desc_k(\bsigma_t)$ with $\bsigma_t \in R_k(\bsigma_s)$.
 Then there are a path $\bsigma_u \to \cdots \to \bsigma_t$ in $(V_k,E_k)$ and
 a path $\bsigma_t \to \cdots \to \bsigma_s$ in $(V_0, E_0)$. Since $(V_k, E_k)$ is
 a subgraph of $(V_0, E_0)$, there is a path from $\bsigma_u$ to $\bsigma_s$
 in $(V_0, E_0)$ through $\bsigma_t$ and this means that $\bsigma_u \in \odesc_k(\bsigma_s)$. To show the inverse inclusion relation, we pick any $\bsigma_u$ in $\odesc_k(\bsigma_s)$.
 Since $\bsigma_u \in V_k$, there is $\bsigma_t \in R_k$ such that
 $\bsigma_u \in \desc_k(\bsigma_t)$. There are a path
 $\bsigma_u \to \cdots \to \bsigma_s \in (V_0, E_0)$ and
 $\bsigma_u \to \cdots \to \bsigma_t \in (V_k, E_k)$. Since $(V_k, E_k)$ is a subgraph
 of $(V_0, E_0)$ and there is a unique path from a node to a root node in a tree,
there is always the node $\bsigma_t$ in the path
 $\bsigma_u \to \cdots \to \bsigma_s \in (V_0, E_0)$ and it means that
$\bsigma_t \in R_k(\bsigma_s)$. We prove
 $\odesc_k(\bsigma_s) = \bigsqcup_{\bsigma_t \in R_k(\bsigma_t)} \desc_k(\bsigma_t)$.

 We will show (iii). $R_k$ is ordered as $\{\bsigma_{s_1}, \ldots \bsigma_{s_m}\}$ with
 $s_1 < \ldots < s_m $.
 Assume that
 \begin{align}
   \sum_{j=1}^m\lambda_j \hat{y}_{s_j}^{(k)} = 0 \label{eq:cohomindep_assumption}
 \end{align}
where $\lambda_j \in \Zint_2$ and we show $\lambda_j = 0$ for all $j$.
Now we consider the equation
$\sum_{j=1}^m\lambda_j \hat{y}_{s_j}^{(k)}(\bsigma_{s_m}) = 0$
by applying \eqref{eq:cohomindep_assumption} to $\bsigma_{s_m}$.
Obviously, $\hat{y}_{s_m}^{(k)}(\bsigma_{s_m}) = 1$ since
$\bsigma_{s_m} \in \odesc_k(\bsigma_{s_m})$ and
$\hat{y}_{s_j}^{(k)}(\bsigma_{s_j}) = 0$ for any $1 \leq j < m$ since
$\bsigma_{s_m} \not \in \odesc_k(\bsigma_{s_j})$ from Fact~\ref{fact:tree_order}.
Therefore we have $\lambda_m = 0$. Repeatedly we can show
$\lambda_{m-1} = \cdots = \lambda_{1} = 0$ in the same way and (iii) is shown.
\end{proof}

The following lemma
is easy to show from the definition of the map.
\begin{lem}\label{lem:include-cohom}
  The map $\bar{\varphi}_k^\vee$ satisfies the following:
  \begin{align*}
    \bar{\varphi}_k^\vee(\hat{y}_s^{(k)}) = \hat{y}_s^{(k+1)}. 
  \end{align*}
\end{lem}

We also show the following lemma.
\begin{lem}\label{lem:pcohom-birthdeath}
  If $(\bsigma_s \xrightarrow{k} \bsigma_t) \in E_0$, the followings hold:
  \begin{enumerate}[(i)]
  \item $\hat{y}_s^{(u)} \not \in Z^0(Y_{k})$ for $u \leq k$
  \item $\hat{y}_s^{(u)} \in Z^0(Y_{k+1})$ for $k+ 1 \leq s$
  \item $\hat{y}_s^{(u)} \not = 0$ for $u \leq s$
  \item $\hat{y}_s^{(u)} = 0$ for $u \geq s+1$
  \end{enumerate}
\end{lem}
\begin{proof}
  Since $\hat{y}_s^{(u)} $ is an element of basis of $Z^0(Y_u)$ due to Lemma~\ref{lem:pcohom-basis}
  for $k+1 \leq u \leq s$ , we have (ii).
  From Fact~\ref{fact:tree_order}, we have
  $\desc_0(\bsigma_s) \subset \{\bsigma_1, \ldots, \bsigma_s\}$ and so
  $\odesc_{s+1}(\bsigma_s) = \emptyset$, therefore (iv) is true.
  Since $\bsigma_s \in \odesc_u(\bsigma_s)$ for any $u \leq s$ from the definition
  of $\odesc_u(\bsigma)$, we have (iii).
  From the theory of 0th cohomology,
  $\hat{y}_s^{(u)} \in Z^0(Y_{k})$ if and only if
  $\odesc_{u}(\bsigma_s)$ is a finite union of connected components.
  However, 
  from Fact~\ref{fact:subtree},
  \begin{align*}
    \odesc_{u}(\bsigma_s) = \desc_{u}(\bsigma_s) \text{ for } u \leq k
  \end{align*}
  from Fact~\ref{fact:graph-cc} this set is a proper subset of $\desc_{u}(\bsigma_v)$
  where $\bsigma_v$ is the root of the tree which has $\bsigma_s$ as a node. Therefore
  we have (i).
\end{proof}

The following theorem is required for the treatment of reduced persistent cohomology.
\begin{lem}\label{lem:cohom-reduced}\ 
  \begin{enumerate}[(i)]
  \item $(V_0, E_0)$ is a single tree.
  \item The root of the single tree is $\bsigma_\infty$
  \item $\bsigma_\infty$ is a root of a tree in $(V_k, E_k)$ for any $k$.
  \item $\hat{y}_{\infty}^{(k)} = \sum_{\bsigma_u:\textrm{0-simplex}, u> k} \bsigma_u^*$
  \item $\tilde{H}^0(Y_k) = H^0(Y_{k}) /\left<\hat{y}_{\infty}^{(k)}\right>$
  \item $\{[y_s^{(k)}]_{\left<\hat{y}_\infty^{(k)}\right>}\mid s\not = \infty,
    \sigma_s \in R_k\}$ and
    $\{[\hat{y}_s^{(k)}]_{\left<\hat{y}_\infty^{(k)}\right>}\mid s\not = \infty,
    \sigma_s \in R_k\}$
    are two bases of $\tilde{H}^0(Y_k)$
  \end{enumerate}
\end{lem}
\begin{proof}
  (i) comes from the connectivity of $Y$ in Condition~\ref{cond:cohom} and
  Fact~\ref{fact:graph-cc}. (ii) and (iii) comes from Fact~\ref{fact:tree_order}.
  (iv) comes from the definition of $\hat{y}_{\infty}^{(k)}$ and (ii).
  (v) comes from (iv) and from the definition of reduced cohomology
  \eqref{eq:cohomology0}.
  Finally, we conclude (vi) by (i-v).
\end{proof}

Lemma~\ref{lem:pcohom-basis}, \ref{lem:include-cohom}, \ref{lem:pcohom-birthdeath},
and \ref{lem:cohom-reduced}
lead Theorem \ref{thm:tree0}. 

\subsection{Merge-tree algorithm for $(n-1)$-th persistent homology}

\begin{proof}[Proof of Theorem~\ref{thm:vochd-alg}(i)]
  We prove Theorem~\ref{thm:vochd-alg}. 
  Theorem~\ref{thm:vochd-alg}(i) is the direct result of
  Theorem~\ref{thm:tree0} and
  $\tilde{H}_{\ell-1}(\X) \simeq \tilde{H}^{n-\ell}(\bar{\X})$
  by applying Algorithm~\ref{alg:tree-0-compute} to
  $\X^+: X_0 \subset X_1 \subset \cdots\subset  X_K $ in $S^n$
  and its dual decomposition 
  $\bar{\X}^+:\bar{X}_0 \supset \cdots \supset \bar{X}_K = \{ \bsigma_\infty\}$.
  To apply Theorem~\ref{thm:tree0}, we check
  $\bar{X}_0 = \{\bsigma_1, \ldots, \bsigma_K, \bsigma_\infty\}$ is connected
  and it is true since the dual decomposition is also a decomposition of $S^n$.
\end{proof}

\begin{proof}[Proofs of Theorem~\ref{thm:vochd-unique} and Theorem~\ref{thm:vochd-alg}(ii)]
  We show that $x_b^{(d)} = \sum_{\bsigma_t \in \desc_b(\bsigma_d)} \sigma_t$
  is a persistent volume for a birth-death pair $(b, d)$. \eqref{eq:vc-1} is shown by
  $\desc_b(\bsigma_d) \subset \{\bsigma_{b+1}, \ldots, \bsigma_d\}$ from Fact~\ref{fact:tree_order} and $\bsigma_d \in \desc_b(\bsigma_d)$ from the definition of $\desc_b(\bsigma_d)$.
  \eqref{eq:vc-2} and \eqref{eq:vc-3} is shown from
  Lemma~\ref{lem:pcohom-birthdeath}(i) and (ii), and \eqref{eq:ccboundary}.

  To prove the optimality of $x_b^{(d)}$, we show the following claim.
  \begin{claim}
    If $x$ is a persistent volume of $(b, d)$, $x_b^{(d)} \subset x$ holds.
  \end{claim}

  The claim immediately leads Theorem~\ref{thm:vochd-unique}
  and Theorem~\ref{thm:vochd-alg}(ii).
  From Theorem~\ref{thm:good_voc}, $[\partial x]_b$ is well defined and
  nonzero. By the isomorphism for Alexander duality and
  Lemma~\ref{lem:cohom-reduced}(vi),
  there is $R \subset R_b$
  such that 
  \begin{align*}
    [\partial x]_b &= \Theta(\sum_{\sigma_s \in R}
                     [y_s^{(b)}]_{\left<\hat{y}_\infty^{(b)}\right>}), \\
    \bsigma_\infty & \not \in R.
  \end{align*}
  From the relation \eqref{eq:ccboundary} and the definition of $x_s^{(b)}$,
  \begin{align*}
    [\partial x_s^{(b)}] = \Theta(
    [y_s^{(b)}]_{\left<\hat{y}_\infty^{(b)}\right>})
  \end{align*}
  for any $s$ with $\sigma_s \in R_k$. Therefore
  \begin{align*}
    [\partial x]_b = \sum_{\sigma_s \in R} [\partial x_s^{(b)}]_b,
  \end{align*}
  and hence
  \begin{align*}
    \partial x + \sum_{\sigma_s \in R} \partial x_s^{(b)} \in B_{n-1}(X_{b}),
  \end{align*}
  so there exists $w \in C_{n}(X_b)$ such that
  \begin{align*}
    \partial(x + \sum_{\sigma_s \in R} x_s^{(b)} + w) = 0
  \end{align*}
  holds. Since $X_b$ is a simplicial complex embedded in $\R^n$, $Z_n(X_b) = 0$ and
  \begin{align*}
    x + \sum_{\sigma_s \in R} x_s^{(b)} + w = 0.
  \end{align*}
  Since $x, x_s^{(b)} \in \left<\sigma_k: n\text{-simplex} \mid b < k \leq d \right>$ and
  $w \in C_n(X_b) = \left<\sigma_k: n\text{-simplex} \mid k < b  \right>$,
  we have $w = 0$ and
  \begin{align*}
    x = \sum_{\sigma_s \in R} x_s^{(b)}.
  \end{align*}
  From \eqref{eq:vc-1}, $x$ has always $\sigma_d$ term and so
  $\sigma_d \in R$ and finish the proof of the claim.
\end{proof}

Theorem~\ref{thm:vochd-tree} and Theorem~\ref{thm:vochd-alg}(iii) are immediately
come from the definition of $x_d^{(b)}$ and properties of the tree structure shown in
Section~\ref{subsec:treemerge-0-pcohom}.

\section*{Acknowledgements}

Dr. Nakamura provided 
the data of the atomic configuration of amorphous silica used in the example.
Dr. Escolar provided the computation software for optimal cycles
on persistent homology. I thank them.

This work is partially supported by 
JSPS KAKENHI Grant Number JP 16K17638, 
JST CREST Mathematics15656429, and
Structural Materials for Innovation, Strategic Innovation Promotion Program
D72.

\bibliographystyle{unsrt}
\bibliography{voc}

\end{document}